\documentclass[a4paper, british]{amsart}

\usepackage{amssymb}
\usepackage{babel}
\usepackage{enumitem}
\usepackage{hyperref}
\usepackage[utf8]{inputenc}
\usepackage{newunicodechar}
\usepackage{varioref}
\usepackage[arrow,curve,matrix]{xy}

\usepackage{colortbl}
\usepackage{graphicx}
\usepackage{tikz}

\usepackage{mathrsfs}

\definecolor{linkred}{rgb}{0.7,0.2,0.2}
\definecolor{linkblue}{rgb}{0,0.2,0.6}

\numberwithin{figure}{section}

\usepackage[hyperpageref]{backref}

\setdescription{labelindent=\parindent, leftmargin=2\parindent}
\setitemize[1]{labelindent=\parindent, leftmargin=2\parindent}
\setenumerate[1]{labelindent=0cm, leftmargin=*, widest=iiii}

\newunicodechar{α}{\ensuremath{\alpha}}
\newunicodechar{β}{\ensuremath{\beta}}
\newunicodechar{χ}{\ensuremath{\chi}}
\newunicodechar{δ}{\ensuremath{\delta}}
\newunicodechar{ε}{\ensuremath{\varepsilon}}
\newunicodechar{Δ}{\ensuremath{\Delta}}
\newunicodechar{η}{\ensuremath{\eta}}
\newunicodechar{γ}{\ensuremath{\gamma}}
\newunicodechar{Γ}{\ensuremath{\Gamma}}
\newunicodechar{ι}{\ensuremath{\iota}}
\newunicodechar{κ}{\ensuremath{\kappa}}
\newunicodechar{λ}{\ensuremath{\lambda}}
\newunicodechar{Λ}{\ensuremath{\Lambda}}
\newunicodechar{μ}{\ensuremath{\mu}}
\newunicodechar{ω}{\ensuremath{\omega}}
\newunicodechar{Ω}{\ensuremath{\Omega}}
\newunicodechar{π}{\ensuremath{\pi}}
\newunicodechar{φ}{\ensuremath{\phi}}
\newunicodechar{Φ}{\ensuremath{\Phi}}
\newunicodechar{ψ}{\ensuremath{\psi}}
\newunicodechar{Ψ}{\ensuremath{\Psi}}
\newunicodechar{ρ}{\ensuremath{\rho}}
\newunicodechar{σ}{\ensuremath{\sigma}}
\newunicodechar{Σ}{\ensuremath{\Sigma}}
\newunicodechar{τ}{\ensuremath{\tau}}
\newunicodechar{θ}{\ensuremath{\theta}}
\newunicodechar{Θ}{\ensuremath{\Theta}}

\newunicodechar{→}{\ensuremath{\to}}
\newunicodechar{⨯}{\ensuremath{\times}}
\newunicodechar{∪}{\ensuremath{\cup}}
\newunicodechar{∩}{\ensuremath{\cap}}
\newunicodechar{⊇}{\ensuremath{\supseteq}}
\newunicodechar{⊃}{\ensuremath{\supset}}
\newunicodechar{⊆}{\ensuremath{\subseteq}}
\newunicodechar{⊂}{\ensuremath{\subset}}
\newunicodechar{≥}{\ensuremath{\geq}}
\newunicodechar{≤}{\ensuremath{\leq}}
\newunicodechar{∈}{\ensuremath{\in}}
\newunicodechar{◦}{\ensuremath{\circ}}
\newunicodechar{°}{\ensuremath{^\circ}}
\newunicodechar{…}{\ifmmode\mathellipsis\else\textellipsis\fi}
\newunicodechar{·}{\ensuremath{\cdot}}

\newunicodechar{¹}{\ensuremath{^1}}
\newunicodechar{²}{\ensuremath{^2}}
\newunicodechar{³}{\ensuremath{^3}}

\DeclareFontFamily{OMS}{rsfs}{\skewchar\font'60}
\DeclareFontShape{OMS}{rsfs}{m}{n}{<-5>rsfs5 <5-7>rsfs7 <7->rsfs10 }{}
\DeclareSymbolFont{rsfs}{OMS}{rsfs}{m}{n}
\DeclareSymbolFontAlphabet{\scr}{rsfs}
\DeclareSymbolFontAlphabet{\scr}{rsfs}

\DeclareMathOperator{\codim}{codim}

\DeclareMathOperator{\Hom}{Hom}

\DeclareMathOperator{\Image}{Image}

\DeclareMathOperator{\Pic}{Pic}
\DeclareMathOperator{\rank}{rank}

\DeclareMathOperator{\reg}{reg}

\DeclareMathOperator{\Sym}{Sym}

\newcommand{\sA}{\scr{A}}
\newcommand{\sB}{\scr{B}}

\newcommand{\sE}{\scr{E}}
\newcommand{\sF}{\scr{F}}
\newcommand{\sG}{\scr{G}}
\newcommand{\sH}{\scr{H}}

\newcommand{\sL}{\scr{L}}

\newcommand{\sN}{\scr{N}}
\newcommand{\sO}{\scr{O}}

\newcommand{\sQ}{\scr{Q}}

\newcommand{\sS}{\scr{S}}
\newcommand{\sT}{\scr{T}}

\newcommand{\bC}{\mathbb{C}}

\newcommand{\bN}{\mathbb{N}}

\newcommand{\bP}{\mathbb{P}}
\newcommand{\bQ}{\mathbb{Q}}
\newcommand{\bR}{\mathbb{R}}

\newcommand{\bZ}{\mathbb{Z}}

\theoremstyle{plain}   
\newtheorem{thm}{Theorem}[section]

\newtheorem{cor}[thm]{Corollary}
\newtheorem{defn}[thm]{Definition} 

\newtheorem{lem}[thm]{Lemma}

\newtheorem{prop}[thm]{Proposition}

\theoremstyle{remark}
\newtheorem{assumption}[thm]{Assumption} 
\newtheorem{asswlog}[thm]{Assumption w.l.o.g.}

\newtheorem{c-n-d}[thm]{Claim and Definition}

\newtheorem{notation}[thm]{Notation}

\newtheorem{rem}[thm]{Remark}
\newtheorem{question}[thm]{Question}
\newtheorem*{rem-nonumber}{Remark}

\numberwithin{equation}{thm}

\setlist[enumerate]{label=(\thethm.\arabic*), before={\setcounter{enumi}{\value{equation}}}, after={\setcounter{equation}{\value{enumi}}}}

\newcommand{\into}{\hookrightarrow}

\newcommand{\wtilde}{\widetilde}
\newcommand{\what}{\widehat}

\def\clap#1{\hbox to 0pt{\hss#1\hss}}

\def\mathclap{\mathpalette\mathclapinternal}

\def\mathclapinternal#1#2{%
\clap{$\mathsurround=0pt#1{#2}$}}

\newcommand\CounterStep{\addtocounter{thm}{1}\setcounter{equation}{0}}

\newcommand{\factor}[2]{\left. \raise 2pt\hbox{$#1$} \right/\hskip -2pt\raise -2pt\hbox{$#2$}}

\makeatletter
\let\saveqed\qed
\renewcommand\qed{%
   \ifmmode\displaymath@qed
   \else\saveqed
   \fi}
\makeatother
\newcommand{\Preprint}[1]{#1}
\newcommand{\Publication}[1]{}

\newcommand{\subversionInfo}{}
\newcommand{\svnid}[1]{}
\newcommand{\approvals}[1]{}

\DeclareMathOperator{\Div}{Div}
\DeclareMathOperator{\Gl}{Gl}
\DeclareMathOperator{\Mov}{Mov}
\DeclareMathOperator{\SStab}{SStab}
\DeclareMathOperator{\Stab}{Stab}

\usepackage{mathpazo}

\title[Movable curves and semistable sheaves]{Movable curves and semistable sheaves}

\author{Daniel Greb}
\author{Stefan Kebekus}
\author{Thomas Peternell}

\address{Daniel Greb, Essener Seminar für Algebraische Geometrie und Arithmetik, Fakultät für Mathematik, Universität Duisburg-Essen, 45117 Essen, Germany}
\email{\href{mailto:daniel.greb@uni-due.de}{daniel.greb@uni-due.de}}
\urladdr{\href{http://www.esaga.uni-due.de/daniel.greb/}{http://www.esaga.uni-due.de/daniel.greb/}}

\address{Stefan Kebekus, Mathematisches Institut, Albert-Ludwigs-Universität
  Freiburg, Eckerstra{\ss}e 1, 79104 Freiburg im Breisgau, Germany, and
  University of Strasbourg Institute for Advanced Study (USIAS), Strasbourg,
  France}
\email{\href{mailto:stefan.kebekus@math.uni-freiburg.de}{stefan.kebekus@math.uni-freiburg.de}}
\urladdr{\href{http://home.mathematik.uni-freiburg.de/kebekus}{http://home.mathematik.uni-freiburg.de/kebekus}}

\address{Thomas Peternell, Mathematisches Institut, Universität
  Bayreuth, 95440~Bayreuth, Germany}
\email{\href{mailto:thomas.peternell@uni-bayreuth.de}{thomas.peternell@uni-bayreuth.de}}
\urladdr{\href{http://btm8x5.mat.uni-bayreuth.de/mathe1}{http://btm8x5.mat.uni-bayreuth.de/mathe1}}

\thanks{All three authors were partially supported by the DFG-Forschergruppe 790
  ``Classification of Algebraic Surfaces and Compact Complex Manifolds''.  Stefan Kebekus acknowledges additional support through a joint fellowship of the Freiburg Institute of Advanced Studies (FRIAS) and the University of
  Strasbourg Institute for Advanced Study (USIAS)}

\definecolor{linkred}{rgb}{0.7,0.2,0.2}
\definecolor{linkblue}{rgb}{0,0.2,0.6}

\makeatletter
\hypersetup{
  pdfauthor={\authors},
  pdftitle={\@title},
  pdfsubject={\@subjclass},
  pdfkeywords={\@keywords},
  pdfstartview={Fit},
  pdfpagelayout={TwoColumnRight},
  pdfpagemode={UseOutlines},
  bookmarks,
  colorlinks}
\makeatother

\date{\today}

\begin{document}

\begin{abstract}
  This paper extends a number of known results on slope-semistable sheaves from
  the classical case to the setting where polarisations are given by movable
  curve classes.  As applications, we obtain new flatness results for reflexive
  sheaves on singular varieties, as well as a characterisation of finite
  quotients of Abelian varieties via a Chern class condition.
\end{abstract}

\maketitle
\tableofcontents

%
%
\svnid{$Id: 01.tex 216 2015-03-10 15:53:08Z peternell $}

\section{Introduction}
\subversionInfo
\approvals{
  Daniel & yes \\
  Stefan & yes \\
  Thomas & yes
}

Given an $n$-dimensional, complex, projective manifold $X$, the notion of
\emph{stability} is arguably the single most important concept in the discussion
of sheaves of $\sO_X$-modules and their moduli.  A widely-used stability notion
is \emph{slope-stability}.  Its well-known definition depends on the choice of a
very ample line bundle $\sH ∈ \Pic(X)$ and asks to associate to any sheaf $\sE $
of positive rank the slope
$$
μ_{\sH}(\sE) := \frac{c_1(\sH)^{n-1} · c_1(\det E)}{\rank \sE} = \frac{[C]
  · c_1(\det E)}{\rank \sE},
$$
where $[C]$ is the numerical class of a curve $C$, obtained as the intersection
of general elements $(H_i)_{1≤i<n} ∈ |\sH|$, that is, $C := H_1 ∩ \cdots ∩
H_{n-1}$.  The sheaf $\sE$ is said to be semistable with respect to $\sH$ if
$\sE$ is torsion free and if $μ_{\sH}(\sF) ≤ μ_{\sH}(\sE)$ for any non-zero
subsheaf $\sF ⊆ \sE$.

It turns out that these notions are often not flexible enough to allow for
applications in higher-dimensional birational geometry.  There, one frequently
needs to discuss a number of birational models, and compare sheaves that live on
one model with sheaves that live on another.  However, the pull-back of an ample
polarisation is generally not ample.  In this setting, it is often advantageous
to generalise the notion of slope, replacing the class $[C]$ with a general
movable curve class.  Recall that a numerical curve class $α$ is movable if $D
· α ≥ 0$ for all effective divisors $D$.  Sample applications in birational
geometry are given in \cite{Miyaoka87, CP11, ExtApplications}.  In a different
direction, the paper \cite{GT13} uses polarisations by movable curves to resolve
pathological wall-crossing phenomena for moduli spaces of sheaves on
higher-dimensional varieties.

The present paper extends a number or known results from the classical case to
the setting where polarisations are given by movable curve classes.  As
applications, we prove new flatness results for reflexive sheaves on singular
varieties, as well as a uniformisation result.

\subsection{Outline of the paper}
\approvals{
  Daniel & yes \\
  Stefan & yes\\
  Thomas & yes
}

We begin in Section~\ref{sect:2} with a brief account of the relevant
definitions and of elementary properties, including the existence of
Harder-Narasimhan and Jordan-Hölder filtrations, boundedness results and a weak
Mehta-Ramanathan theorem.  Once these basics are established, the following
aspects will be discussed in the remainder of the paper.

\subsubsection{Openness of stability}
\approvals{
  Daniel & yes \\
  Stefan & yes\\
  Thomas & yes
}

We will show in Section~\ref{sec:openness} that stability is an open property in
the interior of the movable cone.  Somewhat more generally, we show that if a
sheaf $\sE$ is stable with respect to a movable class $α$ that lies on the
boundary of the movable cone, then $α$ can be approximated by a sequence of
$\sE$-stabilising classes from the interior.

\subsubsection{Semistability of tensor products}
\approvals{
  Daniel & yes \\
  Stefan & yes \\
  Thomas & yes
}

When semistability is defined with respect to an ample class, a classical
theorem (specific to characteristic zero) asserts that the tensor product of two
semistable locally free sheaves is again semistable,
\cite[Thm.~3.1.4]{MR2665168}.  Theorem~\ref{thm:mainthm} of
Section~\ref{sec:tensor} extends this result to sheaves that are semistable with
respect to movable curve classes on possibly singular spaces.

\begin{rem}[Erratum to \cite{ExtApplications}]
  The assertion of Theorem~\ref{thm:mainthm} is stated in \cite{ExtApplications}
  as ``Fact~A.13'', referring to \cite[Thm.~5.1 and Cor.~5.2]{CP11} for a proof.
  However, these references work under the additional assumption that the
  movable class $α ∈ N_1(X)_{\bQ}$ is \emph{big}, that is, contained in the
  interior of the movable cone.  This assumption is not necessarily satisfied in
  the setting of \cite{ExtApplications}.  Theorem~\ref{thm:mainthm} fills this
  gap.
\end{rem}

\subsubsection{Bogomolov-Gieseker inequalities}
\label{ssec:introBGI}
\approvals{
  Daniel & yes \\
  Stefan & yes\\
  Thomas & yes
}
    
Section~\ref{sec:BGI} discusses the Bogomolov-Gieseker inequality for movable
polarisations on a smooth surface: if $\sE$ is a torsion-free coherent sheaf of
rank $r$ on a smooth projective surface, which is semistable with respect to
some non-zero movable curve class, then there is an inequality of Chern numbers,
$2r · c_2(\sE) ≥ (r-1) · c_1²(\sE)$.

\subsubsection{Flatness criteria}
\approvals{
  Daniel & yes \\
  Stefan & yes \\
  Thomas & yes
}

The Bogomolov-Gieseker inequality will be applied in Section~\ref{sec:flatness}
to generalise a flatness criterion of Simpson.  Let $X$ be smooth, complex,
projective surface $X$, let $α$ be a movable class and $\sE$ be a torsion-free
sheaf such that the following numbers vanish,
$$
c_1(\sE) · α = c_1(\sE)² - c_2(\sE) = 0.
$$
If $\sE$ is $α$-semistable and $α² > 0$, then $\sE$ is a locally free, flat
sheaf.  In other words, $\sE$ is given by a linear representation of the
fundamental group.  We also obtain a criterion for projective flatness,
Theorem~\ref{thm:flat:2a}.

The following theorem is a consequence of the flatness criterion and of results
obtained by the authors in \cite{GKP13}.  For convenience, the following
notation is used.  If $f: A → B$ is any morphism of quasi-projective varieties
and if $\sS$ is any coherent sheaf of $\sO_B$-modules, write $f^{[*]} \sS :=
(f^* \sS)^{**}$.  A finite, surjective morphism is called quasi-étale if it is
unbranched in codimension one.

\begin{thm}[Flatness criterion = Theorem~\vref{thm:flat:2b}]\label{thm:1-2}
  Let $X$ be a normal, projective, $\bQ$-factorial variety of dimension $n$ with
  only canonical singularities.  Let $\sE$ be a reflexive sheaf on $X$ and $H ∈
  \Div(X)$ an ample divisor.  Suppose that $\sE$ is $H$-semistable and that
  there exists a desingularisation $π: \wtilde X → X$ such that the following
  two equalities hold,
  $$
  0 = c_1(\sE) · H^{n-1} \quad\text{and}\quad 0 = c_1 \bigl( π^{[*]} \sE
  \bigr)² · \bigl( π^* H \bigr)^{n-2} - c_2 \bigl( π^{[*]} \sE \bigr) ·
  \bigl( π^* H \bigr)^{n-2}.
  $$
  Then, there exists a quasi-étale morphism $γ: \what X → X$ such that $γ^{[*]}
  \sE$ is a locally free, flat sheaf on $\what X$.  \qed
\end{thm}

\begin{rem}[Coefficients in flatness criterion]
  In the setting of Theorem~\ref{thm:1-2}, the Hodge index theorem implies that
  $c_1 \bigl( π^{[*]} \sE \bigr)² · \bigl( π^* H \bigr)^{n-2} ≤ 0$.  Using the
  Bogomolov-Gieseker Inequality of Section~\ref{ssec:introBGI}, one obtains that
  the following conditions, which seemingly depend on the choice of a real
  parameter $λ$,
  $$
  0 = c_1 \bigl( π^{[*]} \sE \bigr)² · \bigl( π^* H \bigr)^{n-2} - λ · c_2
  \bigl( π^{[*]} \sE \bigr) · \bigl( π^* H \bigr)^{n-2},
  $$
  are in fact equivalent for all $0 < λ < \frac{2r}{r-1}$.  The formulation of
  Theorem~\ref{thm:1-2} uses $λ = 1$.  The choice $λ = 2$ is often seen in the
  literature since it appears in the second Chern character,
  cf.~\cite[Cor.~3.10]{MR1179076}.
\end{rem}

\subsubsection{Characterisation of torus quotients}
\approvals{
  Daniel & yes \\
  Stefan & yes\\
  Thomas & yes
}

In the last section we apply the flatness criteria to the cotangent sheaf of a
manifold, in order to obtain the following characterisation result.

\begin{thm}[Characterisation of torus quotients = Theorem~\vref{thm:tquot}]
  Let $X$ be a normal $\bQ$-factorial projective variety of dimension $n$ with
  only canonical singularities and numerically trivial canonical bundle, $K_X
  \equiv 0$.  Assume that there exists a desingularisation $π: \wtilde X → X$
  and an ample divisor $H ∈ \operatorname{Dix}(X)$ such that $c_2(\wtilde X)
  · \bigl( π^*H \bigr)^{n-2} = 0$.  Then, $X$ is smooth in codimension two,
  there exists an Abelian variety $A$ and a quasi-étale morphism $γ: A → X$.
  \qed
\end{thm}

This result generalises a result for three-dimensional varieties of
Shepherd-Barron and Wilson \cite[Cor.~of Main Thm]{SBW94}, and eliminates the a
priori assumption on the codimension of the singular locus made in
\cite[Thm.~1.16]{GKP13}.  After the proof of this result was completed, we
learned that Lu-Taji obtained similar results, to appear in a forthcoming
preprint.

\subsection{Global assumptions}
\approvals{
  Daniel & yes \\
  Stefan & yes\\
  Thomas & yes
}

Throughout the paper, we work over the complex number field.  If not mentioned
otherwise, all sheaves are assumed to be coherent.

\subsection{Acknowledgements}
\approvals{
  Daniel & yes \\
  Stefan & yes\\
  Thomas & yes
}

We would like to thank Matei Toma for important suggestions.  The main idea of
Theorem~\ref{thm:toma} goes back to him.  The authors would also like to thank
two anonymous referees for their helpful and very detailed reports.  A part of
this work was done during a very pleasant stay of the third named author at the
Freiburg Institute for Advanced Studies.

%
%
\svnid{$Id: 02.tex 229 2015-04-15 08:14:24Z kebekus $}

\section{Semistability with respect to a movable class}
\label{sect:2}
\subversionInfo

\subsection{Numerical classes}
\approvals{
  Daniel & yes \\
  Stefan & yes \\
  Thomas & yes
}

Given a normal projective variety $X$, we consider the space $N_1(X)_{\bR}$ of
numerical curve classes, as well as the space $N¹(X)_{\bR}$ of numerical
Cartier divisor classes.  We refer the reader to \cite[Sect.~II.4]{K96} for a
brief definition and discussion of these spaces.  Recall from \cite[Sect.~II.4,
(4.2.5)]{K96} that the intersection number of curves and Cartier divisors gives
rise to a nondegenerate bilinear pairing\CounterStep
\begin{equation}\label{eq:pairing}
  N_1(X)_{\bR} ⨯ N¹(X)_{\bR} → \bR.
\end{equation}

\begin{defn}[Cone of movable curve classes]\label{defn:movable}
  A class $α ∈ N_1(X)_{\bR}$ is called \emph{movable} if $α · D ≥ 0$ for any
  effective Cartier divisor $D$.  The set of movable classes forms a closed,
  convex cone $\Mov(X) ⊂ N_1(X)_{\bR}$, called the \emph{movable cone}.
\end{defn}

\begin{rem}\label{rem:BDPP}
  If $X$ is a manifold, it has been shown in \cite{BDPP} that the movable cone
  is the closure of the convex cone in $N_1(X)_{\bR}$ generated by curves whose
  deformations cover a dense subset of $X$.  A divisor class $Δ ∈ N¹(X)_{\bR}$
  is pseudo-effective if and only if the associated function $\bullet · Δ$ is
  non-negative on $\Mov(X)$.
\end{rem}

\begin{defn}[Big movable class]
  Let $α ∈ N_1(X)_{\bR}$ be a movable class.  We say that $α$ is \emph{big} if
  it lies in the interior of the movable cone.
\end{defn}

\subsubsection{$\bQ$-Cartier divisors on singular spaces}
\approvals{
  Daniel & yes \\
  Stefan & yes \\
  Thomas & yes
}

If $X$ is $\bQ$-factorial, we briefly show that there exists a number $m$ such
that for any Weil divisor $D$, the numerical class $[m· D]$ is the numerical
class of a Cartier divisor.  The following notation will be used.

\begin{defn}[$\bQ$-Cartier divisors that are numerically Cartier]
  Let $X$ be a normal, projective variety.  If $D$ is any $\bQ$-Cartier,
  $\bQ$-Weil divisor on $X$, then $D$ defines a numerical class $[D] ∈
  N¹(X)_{\bQ} ⊂ N¹(X)_{\bR}$.  We say that ``$D$ is numerically Cartier'' if
  there exists a Cartier divisor $Δ$ whose numerical class equals that of $D$,
  that is, $[D] = [Δ]$.
\end{defn}

\begin{defn}[Numerical classes of sheaves]
  Let $X$ be a normal, projective variety.  If $\sF$ is any coherent sheaf on
  $X$, its determinant is a Weil divisorial sheaf, say $\det \sF \cong
  \sO_X(D)$.  If $D$ is $\bQ$-Cartier, we define the numerical class of $\sF$ as
  $[\sF] := [D] ∈ N¹(X)_{\bQ} ⊂ N¹(X)_{\bR}$.  If $X$ is smooth, we also use the
  traditional notation $c_1(\sF) = [\sF]$.
\end{defn}

\begin{thm}[Finite generation of numerical divisor classes]\label{thm:finiteQCart}
  Let $X$ be a normal, $\bQ$-factorial, projective variety.  Then, the set
  $$
  N_{\bQ\Div}(X) := \{ [D] \:|\: D \text{ an integral Weil divisor}\} ⊆ N¹(X)_{\bR}
  $$
  is a lattice which contains the lattice spanned by numerical classes of
  Cartier divisors.  In particular, there exists a positive integer $m ∈ \bQ^+$
  such that $m · D$ is numerically Cartier, for any integral Weil divisor
  $D$ on $X$.
\end{thm}
\begin{proof}
  It suffices to show that $N_{\bQ\Div}(X)$ is finitely generated as a
  $\bZ$-module.  To this end, let $π: \wtilde{X} → X$ be any resolution of
  singularities.  Since $X$ is normal and $\bQ$-factorial, push-forward of Weil
  divisors induces a surjective map
  $$
  π_*: N¹(\wtilde X)_{\bQ} → N¹(X)_{\bQ}.
  $$
  The Theorem of the Base of Néron-Severi, \cite[II Th.~4.5]{K96}, asserts that
  $N¹(\wtilde X)_{\bZ}$ is finitely generated.  Its image under $π_*$ is exactly
  $N_{\bQ\Div}(X)$.
\end{proof}

\subsubsection{Push-forward and pull-back}
\approvals{
  Daniel & yes \\
  Stefan & yes \\
  Thomas & yes
}

Let $φ: X \dasharrow Y$ be any dominant, rational map between normal,
$\bQ$-factorial, projective varieties.  Assume that the inverse $φ^{-1}$ does
not contract any divisors ---examples are given by resolutions of singularities,
more generally birational morphisms, and maps obtained by running the minimal
model program.  In this setting, there are well-known results concerning
push-forward and pull-back.  In the case of rational numerical classes, these
are summarised in \cite{ExtApplications}.  By linearity, everything said there
also holds for real classes.  For the reader's convenience, we recall the most
important facts.

Push-forward of Weil divisors together with the pairing \eqref{eq:pairing}
induces linear morphisms
$$
φ_*: N¹(X)_{\bQ} → N¹(Y)_{\bQ} \quad\text{as well as}\quad φ^*: N_1(Y)_{\bQ} → N_1(X)_{\bQ}.
$$
Since the push-forward of any effective divisor is effective, it follows
immediately from Definition~\ref{defn:movable} that the pull-back of any movable
curve class is again movable.

\begin{prop}[Push-forward and pull-back of sheaves]\label{prop:invariance}
  Let $φ: X → Y$ be any dominant, birational morphism of normal,
  $\bQ$-factorial, projective varieties.  Let $α ∈ N_1(Y)_{\bR}$ be any class,
  and $\sF$, $\sG$ be torsion-free, coherent sheaves on $X$ and $Y$,
  respectively.  Then, the following holds.
  \begin{enumerate}
  \item We have $[\sF] · φ^*α = [φ_* \sF]· α$.

  \item If $\sF$ agrees with $φ^*\sG$ away from the $φ$-exceptional set, and if
    $(r,a) ∈ \bN ⨯ \bZ$ is any pair of numbers, then $\sG$ contains a subsheaf
    $\sG'$ of rank $r$ and intersection $[\sG'] · α = a$ if and only if $\sF$
    contains a subsheaf $\sF'$ of rank $r$ and intersection $[\sF'] · φ^*α = a$
    that agrees with $φ^*\sG'$ away from the $φ$-exceptional set.  \qed
  \end{enumerate}
\end{prop}

\begin{rem}
  Using the notion of (semi)stability with respect to a movable curve class,
  cf.~Definition~\ref{def:semistability} below,
  Proposition~\ref{prop:invariance} asserts that $\sG$ is $α$-(semi)stable if
  and only if $\sF$ is $φ^*α$-(semi)stable.
\end{rem}

\subsection{Slope and semistability}
\approvals{
  Daniel & yes \\
  Stefan & yes \\
  Thomas & yes
}

We will use the following terminology throughout the paper.

\begin{defn}[Slope with respect to a movable class]\label{def:slope}
  Let $X$ be a normal, $\bQ$-factorial, projective variety and $α ∈ \Mov(X)$.
  If $\sE \ne 0$ is any torsion-free, coherent sheaf on $X$, one defines the
  \emph{slope of $\sE$ with respect to $α$} as the real
  number
  $$
  μ_{α}(\sE) := \frac{ [ \sE ] · α}{\rank \sE}.
  $$
\end{defn}

\begin{defn}[Semistability with respect to a movable class]\label{def:semistability}
  In the setting of Definition~\ref{def:slope}, we say that $\sE$ is
  \emph{α-semistable} if $μ_{α}(\sF) ≤ μ_{α}(\sE)$ for any coherent subsheaf $0
  \subsetneq \sF ⊆ \sE$.
\end{defn}

\begin{defn}[Stability with respect to a movable class]\label{def:stability}
  In the setting of Definition~\ref{def:slope}, we say that $\sE$ is
  \emph{α-stable} if $μ_{α}(\sF) < μ_{α}(\sE)$ for any coherent subsheaf $0
  \subsetneq \sF \subsetneq \sE$ with $\rank \sF < \rank \sE$.
\end{defn}

\subsection{Elementary properties}
\approvals{
  Daniel & yes \\
  Stefan & yes \\
  Thomas & yes
}

Essentially all elementary properties satisfied by sheaves that are stable with
respect to ample polarisations also hold when stability is defined by a movable
class.  For the reader's convenience, we summarise those properties that will be
relevant later.

\begin{prop}[Subbundle of equal rank]
  Let $X$ be a normal, $\bQ$-factorial, projective variety and $α ∈ \Mov(X)$.
  If $\sF ⊆ \sE$ are two torsion-free, coherent sheaves of equal rank, then
  $μ_{α}(\sF) ≤ μ_{α}(\sE)$.
\end{prop}
\begin{proof}
  There exists an effective Weil-divisor $D$ such that $\det \sF$ and $\det \sE$
  differ only by a twist with $D$, or more precisely, $\det \sF \cong \bigl(
  \sO_X(-D) \otimes \det \sE \bigr)^{**}$.  In particular, we have the following
  equality of numerical classes,
  $$
  [\sF] = [\sE] - [D].
  $$
  Since α is movable, we have that $[D] · α ≥ 0$, and the claim follows.
\end{proof}

\begin{cor}[Slope of saturation]\label{cor:slopesat}
  Let $X$ be a normal, $\bQ$-factorial, projective variety and $α ∈ \Mov(X)$.
  If $\sF ⊆ \sE$ are two torsion-free, coherent sheaves, and if
  $\sF_{\operatorname{sat}} ⊆ \sE$ denotes the saturation of $\sF$ in $\sE$,
  then $μ_{α}(\sF) ≤ μ_{α}(\sF_{\operatorname{sat}})$.  \qed
\end{cor}

\begin{cor}[Stable sheaves are semistable]
  Let $X$ be a normal, $\bQ$-factorial, projective variety and $α ∈ \Mov(X)$.
  Then, α-stable sheaves are α-semistable.  \qed
\end{cor}

\begin{cor}[Semistability of direct sum of line bundles]\label{cor:sstabds}
  Let $X$ be a normal, $\bQ$-factorial, projective variety.  If $\sH ∈ \Pic(X)$
  is any line bundle and $r ∈ \bN^+$ any number, then $\sH^{\oplus r}$ is
  semistable with respect to any movable class.  \qed
\end{cor}

The following proposition follows from elementary Chern class computations,
which we leave to the reader.

\begin{prop}[Morphism from a semistable sheaf]\label{prop:mfss}
  Let $X$ be a normal, $\bQ$-factorial, projective variety and $α ∈ \Mov(X)$.
  If $\sF$ is any α-semistable sheaf and $γ : \sF → \sE$ any morphism of
  torsion-free $\sO_X$-modules, then $μ_{α}(\Image γ) ≥ μ_{α}(\sF)$.  \qed
\end{prop}

\begin{cor}[Morphisms between semistable sheaves]\label{cor:msss}
  Let $X$ be a normal, $\bQ$-factorial, projective variety and $α ∈ \Mov(X)$.
  \begin{itemize}
  \item If $\sF$ and $\sE$ are semistable and if $μ_{α}(\sF) > μ_{α}(\sE)$, then
    $\Hom_{\sO_X} \bigl(\sF, \,\sE \bigr) = 0$.

  \item If $\sF$ and $\sE$ are stable and if $μ_{α}(\sF) = μ_{α}(\sE)$, then any
    non-zero morphism $\sF → \sE$ is injective and generically isomorphic.  If
    we assume in addition that $\sF$ is saturated\footnote{Recall that $\sF$ is
      said to be saturated in $\sE$ if the quotient $\sE/\sF$ is torsion-free.}
    in $\sE$, then any non-zero morphism $\sF → \sE$ is an isomorphism. \qed
  \end{itemize}
\end{cor}

\subsection{A generalisation of Mehta-Ramanathan's theorem}
\approvals{
  Daniel & yes \\
  Stefan & yes \\
  Thomas & yes
}

This section contains a minor generalisation of the classical theorem of
Mehta-Ramanathan.

\begin{prop}\label{MRbig}
  Let $π: \wtilde X → X$ be a birational morphism of normal, $\bQ$-factorial,
  projective varieties of dimension $n ≥ 2$ with only canonical singularities.
  Let $H$ be an ample divisor on $X$ and $\wtilde H := π^* H$.  Let $\wtilde
  \sE$ be a torsion-free, $\wtilde H$-semistable sheaf on $\wtilde X$.  If $m
  \gg 0$ is sufficiently large and $\wtilde D_1, …, \wtilde D_{n-2} ∈ |m \wtilde
  H|$ is a general $(n-2)$-tuple of hypersurfaces with associated complete
  intersection surface $\wtilde S =\wtilde D_1 ∩ … ∩ \wtilde D_{n-2}$, then
  $\wtilde \sE|_{\wtilde S}$ is $\wtilde H|_{\wtilde S}$-semistable.
\end{prop}

\begin{rem}
  In the setting of Proposition~\ref{MRbig}, recall from \cite[Prop.~5.1]{GKP13}
  that the sheaf $\wtilde \sE|_{\wtilde S}$ is torsion-free.  The surface
  $\wtilde S$ has canonical singularities, \cite[Lem.~5.7]{KM98}, and is
  therefore $\bQ$-factorial, \cite[Prop.~4.11]{KM98}.  The assertion that
  $\wtilde \sE|_{\wtilde S}$ is $\wtilde H|_{\wtilde S}$-semistable therefore
  makes sense.
\end{rem}

\begin{proof}[Proof of Proposition~\ref{MRbig}]
  Observe that the torsion-free coherent sheaf $\sE := π_*(\wtilde \sE)$ is
  $H$-semistable by Proposition~\ref{prop:invariance}.  Next, write $\wtilde D_i
  = π^*(D_i)$ and set $S := D_1 ∩ … D_{n-2}$ so that $\wtilde S = π^{-1}(S)$.
  Since $X$ has only canonical singularities, conclude as before that $S$ is
  $\bQ$-factorial.  Also, observe that $\sE|_S$ is torsion-free.  By Flenner's
  version of the Mehta-Ramanathan theorem, \cite[Thm.~1.2]{Flenner84}, the
  torsion free sheaf $\sE|_S$ is $H|_S$-semistable.  Since $\wtilde
  \sE|_{\wtilde S}$ and $(π|_{\wtilde S})^*(\sE|_S)$ agree outside the
  exceptional locus of $π|_{\wtilde S}$, Proposition~\ref{prop:invariance}
  applies, showing that $\wtilde \sE|_{\wtilde S}$ is $\wtilde H|_{\wtilde
    S}$-semistable.
\end{proof}

\subsection{Boundedness I: Suprema of slopes in a given bundle}
\label{ssec:ssgb}
\approvals{
  Daniel & yes \\
  Stefan & yes \\
  Thomas & yes
}

Given a torsion-free sheaf $\sE$ and a movable class $α$, the $α$-slope of
subsheaves $\sF ⊆ \sE$ cannot be arbitrarily large.  The following boundedness
results will be used later.

\begin{defn}[Suprema of slopes]\label{def:mumax}
  Let $X$ be a normal, $\bQ$-factorial, projective variety and $α ∈ \Mov(X)$.
  If $\sE$ is any torsion-free, coherent sheaf of positive rank on $X$, write
  $$
  μ_{α}^{\max} (\sE) := \sup \left\{ μ_{α}(\sF) \,|\, \text{$ 0 \ne \sF ⊆ \sE$ a coherent subsheaf} \right\}.
  $$
\end{defn}

\begin{prop}[Boundedness of $μ_{α}^{\max}$]\label{prop:bounded}
  In the setting of Definition~\ref{def:mumax}, the function
  $$
  m : \Mov(X) → \bR ∪ \{ \infty \}, \qquad β \mapsto μ_{β}^{\max}(\sE)
  $$
  is bounded from above by a linear function $M$.  In particular, $μ_{α}^{\max}
  (\sE) < \infty$.
\end{prop}
\begin{proof}
  Choose a sufficiently ample line bundle $\sH$ and an embedding $\sE \into
  \sH^{\oplus r}$, where $r := \rank \sE$.  If $\sF ⊆ \sE$ is any coherent
  subsheaf, it follows from Corollary~\ref{cor:sstabds} that $μ_{α}(\sF) ≤
  μ_{α}(\sH^{\oplus r}) = [\sH] · α$ for every $α ∈ \Mov (X)$.  We can therefore
  take $M(β) := [\sH] · β$ as the desired linear function.
\end{proof}

\begin{prop}[Existence of subsheaves with maximal slope, I]\label{prop:eswms}
  In the setting of Definition~\ref{def:mumax}, the supremum $μ_{α}^{\max}$ is a
  maximum.  In other words, there exists a non-zero, coherent subsheaf $\sF ⊆
  \sE$ such that $μ_{α}^{\max} (\sE) = μ_{α}(\sF) < \infty$.
\end{prop}
\begin{proof}
  Argue by contradiction, and assume that $μ_{α}(\sF) < μ_{α}^{\max} (\sE)$ for
  any non-zero, coherent subsheaf $\sF ⊆ \sE$.  One can then find a sequence of
  subsheaves $(\sF_i)_{i ∈ \bN}$ such that the following holds.
  \begin{enumerate}
  \item\label{il:scms1} The sequence of slopes converges, $\lim μ_{α}(\sF_i) =
    μ_{α}^{\max} (\sE)$.
  \item\label{il:scms2} The sheaves $\sF_i$ are saturated in $\sE$.  All $\sF_i$
    have the same rank $r$.
  \end{enumerate}
  In addition, we can assume that the rank $r$ is maximal among all sequences of
  sheaves satisfying \ref{il:scms1}--\ref{il:scms2}.  We will construct a
  contradiction by proving the following.

  \begin{quotation}
    Given any number $ε > 0$, there exists a subsheaf $\sG_{ε} ⊆ \sE$ such that
    $μ_{α}(\sG_{ε}) ≥ μ_{α}^{\max} (\sE) - ε$ and such that $\rank \sG_{ε} > r$.
  \end{quotation}

  \smallskip

  To this end, let $ε$ be any given number.  Then, there exists an index $i$
  such that $μ_{α}(\sF_i)
  > μ_{α}^{\max} (\sE) - \frac{ε}{2}$.  There exists a larger index $j > i$ such
  that $μ_{α}(\sF_i) < μ_{α}(\sF_j)$.  The subsheaves $\sF_i$ and $\sF_j$ are
  certainly not equal.  Since both $\sF_i$ and $\sF_j$ are saturated in $\sE$,
  their sum $\sG_{ε} := \sF_i + \sF_j$ therefore has $\rank \sG_{ε} > r$.  A
  rather elementary computation of Chern classes, spelled out in
  \cite[Lemma~A.12]{ExtApplications}, now shows that $μ_{α} \bigl( \sG_{ε}
  \bigr) > μ_{α}^{\max} (\sE) - ε$.
\end{proof}

\begin{rem}
  Any sheaf $\sF ⊆ \sE$ with $μ_{α}^{\max} (\sE) = μ_{α}(\sF)$ is clearly
  semistable.
\end{rem}

\subsection{Existence of the Harder-Narasimhan-filtration}
\approvals{
  Daniel & yes \\
  Stefan & yes \\
  Thomas & yes
}

One consequence of the boundedness result obtained in
Proposition~\ref{prop:eswms} is the existence of a maximally destabilising
subsheaf, which in turn implies the existence of the Harder-Narasimhan and of
Jordan-Hölder-filtrations, even in cases where slope is defined by a movable
curve class.

\begin{cor}[Existence of a unique maximally destabilising sheaf]\label{cor:maxdest}
  In the setting of Definition~\ref{def:mumax}, there exists a sheaf $\sF ⊆ \sE$
  such that the following holds.
  \begin{enumerate}
  \item The slope is maximal, $μ_{α}(\sF) = μ_{α}^{\max} (\sE)$.
  \item If $\sF' ⊆ \sE$ is any other subsheaf with $μ_{α}(\sF) = μ_{α}^{\max}
    (\sE)$, then $\sF' ⊆ \sF$.
  \end{enumerate}
  The sheaf $\sF$ is called ``maximally destabilising subsheaf''.  It is clearly
  unique, semistable, and saturated in $\sE$.
\end{cor}
\begin{proof}
  By Proposition~\ref{prop:eswms}, there exists a saturated sheaf $\sF_1 ⊆ \sE$
  of maximal slope.  If $\sF'$ is any other subsheaf of maximal slope, then
  either $\sF'$ is contained in $\sF_1$, or it follows from
  Proposition~\ref{prop:mfss} that the $α$-slope of $\sF_1 + \sF'$ equals
  $μ_{α}^{\max} (\sE)$.  In the second case, set $\sF_2 := (\sF_1 +
  \sF')_{\operatorname{sat}}$ and observe that $\rank \sF_2 > \rank \sF_1$.
  Iterate this process, in order to construct a strictly increasing sequence of
  sheaves of maximal slope, $\sF_1 \subsetneq \sF_2 \subsetneq \dots ⊆ \sE$.
  The process terminates because the rank increases in each step.
\end{proof}

\begin{cor}[Existence of a stable destabilising sheaf]
  In the setting of Definition~\ref{def:mumax}, there exists an α-stable sheaf
  $\sF ⊆ \sE$ of slope $μ_{α}(\sF) = μ_{α}^{\max} (\sE)$.
\end{cor}
\begin{proof}
  By Proposition~\ref{prop:eswms}, there exists a saturated, $α$-semistable
  sheaf $\sF_1 ⊆ \sE$ of maximal slope.  If $\sF_1$ is not stable, there exists
  a sheaf $\sF_2 \subsetneq \sF_1$ that is also of maximal slope, but of smaller
  rank: $\rank \sF_2 < \rank \sF_1$.  Iterate this process, in order to
  construct a strictly decreasing sequence of sheaves of maximal slope,
  $\sF_1 \supsetneq \sF_2 \supsetneq \cdots$.  The process terminates because
  the rank decreases in each step.
\end{proof}

\begin{cor}[Existence of the Harder-Narasimhan-filtration]
  In the setting of Definition~\ref{def:mumax}, there exists a unique
  ``Harder-Narasimhan-filtration'', that is, a filtration
  $0 = \sE_0 \subsetneq \sE_1 \subsetneq \cdots \subsetneq \sE_r = \sE$ where
  each quotient $\sQ_i := \factor{\sE_i}{\sE_{i-1}}$ is torsion-free,
  α-semistable, and where the sequence of slopes $μ_{α} \bigl( \sQ_i \bigr)$ is
  strictly decreasing.  \qed
\end{cor}

\begin{cor}[Existence of Jördan-Hölder-filtrations]\label{cor:JH}
  In the setting of Definition~\ref{def:mumax}, if $\sE$ is α-semistable, then
  there exists a ``Jordan-Hölder-filtration'', that is, a filtration
  $0 = \sE_0 \subsetneq \sE_1 \subsetneq \cdots \subsetneq \sE_r = \sE$ where
  each quotient $\sQ_i := \factor{\sE_i}{\sE_{i-1}}$ is torsion-free, α-stable,
  and with slopes $μ_{α} \bigl( \sQ_i \bigr) = μ_{α}\bigl( \sE \bigr)$.  \qed
\end{cor}

\begin{rem}[Refined Harder-Narasimhan-filtration]\label{rem:RHNF}
  In the setting of Definition~\ref{def:mumax}, combining Harder-Narasimhan and
  Jordan-Hölder filtrations, one obtains a ``refined
  Harder-Narasimhan-filtration''
  $0 = \sE_0 \subsetneq \sE_1 \subsetneq \cdots \subsetneq \sE_r = \sE$ where
  each quotient $\sQ_i := \factor{\sE_i}{\sE_{i-1}}$ is torsion-free, α-stable,
  and where the sequence of slopes $μ_{α}\bigl( \sQ_i \bigr)$ is decreasing
  (though not necessarily strictly decreasing).  \qed
\end{rem}

\subsection{Boundedness II: Grothendieck's lemma}
\approvals{
  Daniel & yes \\
  Stefan & yes \\
  Thomas & yes
}

Let $X$ be a projective manifold and $\sE$ be any torsion-free sheaf on $X$.  If
$H$ any ample Cartier divisor with associated movable curve class $β := H^{\dim
  X-1}$ and $c$ is any real number, then the classical Grothendieck lemma,
\cite[Lem.~1.7.9]{MR2665168}, asserts that the set of subsheaves with bounded
slope,
$$
S := \{ \, \sF \, \:|\: \sF ⊆ \sE \text{ saturated, positive-rank with }
μ_{β}(\sF) ≥ c \}
$$
forms a bounded family.  In particular, the associated set of numerical classes
is finite.  We show that this conclusion still holds in case where $β$ is an
arbitrary big class.

\begin{thm}[Grothendieck's lemma for numerical classes]\label{thm:numGrothendieck}
  Let $X$ be a normal, projective, $\bQ$-factorial variety and $β ∈ \Mov(X)$ be
  a big class.  Further, let $\sE$ be a torsion-free, coherent sheaf on $X$ and
  $c ∈ \bR$ be a real number.  Then, the following set of numerical classes,
  $$
  S_{\operatorname{num}} := \{ \, [\sF] \, \:|\: \sF ⊆ \sE \text{ any
    positive-rank subsheaf with } μ_{β}(\sF) ≥ c \} ⊆ N¹(X)_{\bQ},
  $$
  is finite.
\end{thm}
\begin{proof}
  As before, choose a sufficiently ample bundle $\sH$ and an embedding $\sE
  \into \sH^{\oplus r}$, where $r := \rank \sE$.  It will then suffice to show
  the claim in case where $\sE = \sH^{\oplus r}$.  Since slopes behave linearly
  under twist, we may even assume without loss of generality that $\sE =
  \sO_X^{\oplus r}$.

  Now, given any subsheaf $\sF ⊆ \sE$ of positive rank, its determinant $\det
  \sF$ embeds into $\sO_X$ and is therefore a Weil divisorial sheaf of the form
  $\det \sF \cong \sO_X(-D)$, where $D$ is an effective, integral Weil divisor.
  Since $β$ is assumed to be a big class, it will therefore intersect $[\sF]$
  negatively, unless $[\sF] = 0$.  The assertion is thus a consequence of
  Theorem~\ref{thm:finiteQCart}.
\end{proof}

%
%
\svnid{$Id: 03.tex 226 2015-03-17 07:46:27Z kebekus $}

\section{Openness of semistability}
\subversionInfo
\label{sec:openness}
\approvals{
  Daniel & yes \\
  Stefan & yes \\
  Thomas & yes }

We now show that stability is an open property, at least within the interior of
the movable cone.  More precisely, we discuss openness properties for the set of
stabilising classes, defined as follows.

\begin{defn}[Stabilising classes]\label{def:sc}
  Let $X$ be a normal, $\bQ$-factorial, projective variety and let $\sE$ be any
  non-trivial, torsion-free sheaf on $X$.  We consider the set of movable
  classes that (semi-)stabilise $\sE$,
  \begin{align*}
    \Stab(\sE) & := \{ α ∈ \Mov(X) \:|\: \sE \text{ is α-stable}\} \\
    \SStab(\sE) & := \{ α ∈ \Mov(X) \:|\: \sE \text{ is α-semistable}\}.
  \end{align*}
\end{defn}

\begin{rem}[Convexity of stabilising classes]\label{rem:cvx}
  The sets $\Stab(\sE)$ and $\SStab(\sE)$ of Definition~\ref{def:sc} are clearly
  convex.  The set $\SStab(\sE)$ is closed and contains $\Stab(\sE)$.
\end{rem}

\begin{thm}[Openness of stability, I]\label{thm:openness}
  Let $X$ be a normal, $\bQ$-factorial, projective variety and let $\sE$ be any
  non-trivial, torsion-free sheaf on $X$.  If $α ∈ \Stab(\sE)$ is big, then
  $\Stab(\sE)$ contains an open neighbourhood $U = U(α) ⊆ \Mov(X)$.
\end{thm}

The following result asserts that even in cases where we cannot show openness,
any class $α ∈ \Stab(\sE)$ can be approximated by big classes in $\Stab(\sE)$.
The authors would like to thank Matei Toma who explained Theorem~\ref{thm:toma}
to us in case where $β$ is a general complete intersection curve.  Together with
further related results, his argument will also appear in \cite{GRT}.

\begin{thm}[Openness of stability, II]\label{thm:toma}
  Let $X$ be a normal, $\bQ$-factorial, projective variety and let $\sE$ be any
  non-trivial, torsion-free sheaf on $X$.  Consider classes $α ∈ \Stab(\sE)$ and
  $β ∈ \Mov(X)°$.  Then, there exists a number $e ∈ \bQ^+$ such that $(α+ε· β) ∈
  \Stab(\sE)$, for any real $ε ∈ [0,e]$.  In particular, if $L ⊂ N_1(X)_{\bR}$
  is any line through $α$ that intersects the interior $\Mov (X)°$, then $L ∩
  \Stab(\sE)$ is an interval of positive length.
\end{thm}

\begin{rem}\label{rem:toma}
  Let $X$ be a normal, $\bQ$-factorial, projective variety and let $\sE$ be any
  non-trivial, torsion-free sheaf on $X$.  Given any class $α ∈ \Stab(\sE)$,
  Theorems~\ref{thm:openness} and \ref{thm:toma} together assert that
  $\Stab(\sE) ∩ \Mov(X)°$ is open and has $α$ as a boundary point.  In
  particular, there exists a sequence of big, rational classes $β_i ∈
  \Stab(\sE)$ with $\lim β_i = α$.
\end{rem}

We do not know if these results are optimal.  \Preprint{For instance, we cannot
  rule out that the set of movable classes that stabilise a given bundle $\sE$
  is of the form illustrated in Figure~\vref{fig:hyp_Estab}.  } For all
applications, Theorems~\ref{thm:openness} and \ref{thm:toma} seem to suffice.

\Preprint{
\begin{figure}

\centering
\footnotesize

\begin{tikzpicture}
\draw [fill=gray!10] (0,0) circle [x radius=3, y radius=1.5];
\draw (3,1.5) node[above] {Cross-section through Mov($X$)};
\draw [dashed, fill=gray!20] (-1,0) circle [x radius=2, y radius=1] node[above]{$\Stab(\sE)$};
\filldraw (-3,0) circle (0.05) node[left] {class $α$} to (1.5,-1) circle (0.05) node[right] {$β$};
\end{tikzpicture}

\bigskip

{\small The set of movable classes that stabilise a given bundle $\sE$ is open
  in the interior of the movable cone, but does contain an isolated class on the
  boundary.  Note that Theorem~\ref{thm:toma} holds in this context because the
  boundary of $\Stab(\sE)$ intersects the boundary of the movable cone
  tangentially.}

\caption{Hypothetical cross-section through $\operatorname{Mov}(X)$.}
\label{fig:hyp_Estab}
\end{figure}
}

\begin{question}
  Are there examples where $\Stab(\sE)$ is not open in $\Mov(X)$?  If so, are
  there natural conditions to guarantee openness?
\end{question}

We begin with a preparatory subsection.  Theorems~\ref{thm:openness} and
\ref{thm:toma} are then shown in
Sections~\ref{pf:thm:openness}--\ref{pf:thm:toma} below.

\subsection{Suprema of slopes of strict subsheaves}
\approvals{
  Daniel & yes \\
  Stefan & yes \\
  Thomas & yes
}

Given a sheaf $\sE$ and a movable class $α$, we have discussed the number
$μ_{α}^{\max} (\sE)$ in Section~\ref{ssec:ssgb}.  Here, we discuss a similar
(but more delicate) quantity, namely the supremum of slopes of subsheaves $\sF
\subsetneq \sE$ whose rank is strictly smaller than that of $\sE$.  If $α$ is
either big or rational, we will again see that the supremum is in fact a
maximum.

\begin{defn}[Suprema of slopes of strict subsheaves]\label{def:mumax2}
  Let $X$ be a normal, $\bQ$-factorial, projective variety and $α ∈ \Mov(X)$ a
  movable class.  If $\sE$ is any torsion-free, coherent sheaf of
  $\sO_X$-modules with $\rank \sE ≥ 2$, write
  $$
  μ_{α}^{\max,sc} (\sE) := \sup \left\{ μ_{α}(\sF) \,|\, \text{$ 0 \ne \sF \subsetneq
      \sE$ coherent with } \rank \sF < \rank \sE \right\}.
  $$
\end{defn}

\begin{rem}
  Obviously, $μ_{α}^{\max,sc} (\sE) ≤ μ_{α}^{\max} (\sE)$.
\end{rem}

\begin{prop}[Existence of subsheaves with maximal slope, II]\label{prop:maxsc}
  In the setting of Definition~\ref{def:mumax2}, suppose that $α$ is big or that
  $α$ is rational.  Then, there exists a coherent subsheaf $\sF ⊆ \sE$ such that
  $\rank \sF < \rank \sE$ and such that $μ_{α}^{\max,sc}(\sE) = μ_{α}(\sF)$.
\end{prop}
\begin{proof}[Proof of Proposition~\ref{prop:maxsc} in case where $α$ is big]
  Observing that the statement of Proposition~\ref{prop:maxsc} remains invariant
  when twisting $\sE$, we may replace $\sE$ by a tensor product and assume that
  there exists a number $N ∈ \bN^+$, an ample line bundle $\sH$ and a surjection
  $\sH^{\oplus N} → \sE$.  In particular, given any torsion-free quotient $\sE →
  \sQ$ of positive rank, there exists a non-trivial morphism of Weil divisorial
  sheaves
  $$
  \sH^{\otimes \rank \sQ} → \det \sQ.
  $$
  It follows that the numerical class $[\sQ]$ is pseudo-effective, and in fact
  big.

  Introduce any norm $\Vert · \Vert$ on the finite-dimensional space
  $N¹(X)_{\bR}$ of numerical Cartier divisor classes.  Since $α$ is big, there
  exists a constant $C > 0$ such that
  \begin{equation}\label{eq:dfl3}
    D · α ≥ C · \Vert D \Vert \quad \text{ for any pseudo-effective } D ∈ N¹(X)_{\bR}.
  \end{equation}
  In particular, Inequality~\eqref{eq:dfl3} holds for divisors $D$ that
  represent numerical classes of non-trivial, torsion-free quotients of $\sE$.

  Returning to the assertion of Proposition~\ref{prop:maxsc}, we argue by
  contradiction and assume that the number $μ_{α}^{\max,sc} (\sE)$ is not
  attained.  Then, there exists a sequence of subsheaves $\sF_j \subsetneq \sE$
  such that the sequence of slopes $μ_{α}(\sF_j)$ is strictly increasing and
  converges to $μ_{α}^{\max,sc}(\sE)$.  Corollary~\ref{cor:slopesat} allows to
  assume that the quotients $\sQ_j := \sE/\sF_j$ are torsion-free.

  The assumption that $μ_{α}(\sF_j)$ is strictly increasing implies that the
  sets of numerical classes $\bigl\{ [\sF_j] \:|\: j ∈ \bN\bigr\}$ and $\bigl\{
  [\sQ_j] \:|\: j ∈ \bN\bigr\}$ are both infinite.  The theorem on finite
  generation of numerical divisor classes, Theorem~\ref{thm:finiteQCart}, thus
  implies that both sets are unbounded with respect to the norm $\Vert ·
  \Vert$.  Inequality~\eqref{eq:dfl3} thus implies that the sequence $\bigl(
  μ_{α}(\sQ_j) \bigr)_{j ∈ \bN}$ is unbounded, and so is $\bigl( μ_{α}(\sF_j)
  \bigr)_{j ∈ \bN}$.  This contradicts convergence.
\end{proof}
\begin{proof}[Proof of Proposition~\ref{prop:maxsc} in case where $α$ is rational]
  If $α$ is rational, the Theorem on finite generation of numerical divisor
  classes, Theorem~\ref{thm:finiteQCart}, allows to find an integer $m ∈ \bN^+$
  such that $(m · α) · [\sF]$ is integral, for any subsheaf $\sF ⊆ \sE$
  of positive rank (and in fact for any coherent sheaf on $X$).  It follows that
  the slope $μ_{α}(\sF)$ takes values in the discrete set $\frac{1}{ (\rank
    \sE)!  · m} · \bZ$ and the claim is obvious.
\end{proof}

\begin{rem}
  Let $α$ be a movable class that is either big or rational.  As a consequence
  of Proposition~\ref{prop:maxsc}, we see that $\sE$ is $α$-stable if and only
  if $μ_{α}^{\max,sc}(\sE) < μ_{α}^{\max}(\sE)$.
\end{rem}

\begin{rem}
  There are relevant situations where the conclusion of
  Proposition~\ref{prop:maxsc} holds even for irrational classes that lie on the
  boundary of $\Mov(X)$.  The most significant is that where $π: \wtilde X → X$
  is a resolution of singularities and $α ∈ \Mov(X)$ a big, movable class.  Its
  pull-back $π^*α$ is a movable class contained in the boundary $\partial
  \Mov(\wtilde X)$.  However, Proposition~\ref{prop:invariance} guarantees that
  everything said so far also holds for $φ^*α$.
\end{rem}
  
\begin{rem}
  We do not know whether Proposition~\ref{prop:maxsc} remains true for
  irrational classes $α$ that lie on the boundary of the movable cone.
\end{rem}

\subsection{Proof of Theorem~\ref*{thm:openness}}
\label{pf:thm:openness}
\approvals{
  Daniel & yes \\
  Stefan & yes \\
  Thomas & yes }

Observe that if $\rank \sE = 1$, then there is nothing to show.  We will
therefore assume throughout that proof that $\rank \sE ≥ 2$.  Choose any
relatively compact, open neighbourhood of $α$ in $\Mov(X)$, say $V = V(α)$.
Then, there exists a (small) number $e ∈ \bQ^+$ such that the following holds
for all numbers $0 ≤ ε ≤ e$ and for all classes $β ∈ V$,
\begin{equation}\label{eq:B}
  ε · \bigl( \underbrace{μ^{\max}_{β}(\sE) - μ_{β}(\sE)}_{\mathclap{\text{bounded on $V$ by Prop.~\ref{prop:bounded}}}} \bigr) < (1-ε) · \bigl( \underbrace{μ_{α\vphantom{β}}(\sE) - μ^{\max,sc}_{α}(\sE)}_{> 0 \text{ by Prop.~\ref{prop:maxsc}}} \bigr).
\end{equation}
The following will then hold for any class $β ∈ V$, any coherent subsheaf $\sF
\subsetneq \sE$ with $\rank \sF < \rank \sE$, and any number $0 < ε ≤ e$,
\begin{align*}
  μ_{(1-ε) · α + ε · β}(\sF) & = (1-ε) · μ_{α}(\sF) + ε · μ_{β}(\sF) && \\
  & ≤ (1-ε) · μ^{\max,sc}_{α}(\sE) + ε · μ^{\max}_{β}(\sE) && \\
  & < (1-ε) · μ_{α}(\sE) + ε · μ_{β}(\sE) && \text{Equation~\eqref{eq:B}} \\
  & = μ_{(1-ε) · α + ε · β}(\sE).
\end{align*}
Finish the proof of Theorem~\ref{thm:openness} by setting $U := e · (V-α) +
α$.  \qed

\subsection{Proof of Theorem~\ref*{thm:toma}}
\label{pf:thm:toma}
\approvals{
  Daniel & yes \\
  Stefan & yes \\
  Thomas & yes
}

Consider the set $S$ of classes $[\sF]$ of sheaves $\sF ⊆ \sE$ that destabilise
$\sE$ with respect to $β$, i.e., $μ_{β}(\sF) > μ_{β}(\sE)$.  By Grothendieck's
lemma for numerical classes, Theorem~\ref{thm:numGrothendieck}, the set $S$ is
finite, say $S= \bigl\{ [\sF_1], …, [\sF_n] \bigr\}$ for suitable $\sF_i ⊆ \sE$.
Now, if $\sF ⊂ \sE$ is any subsheaf with $\rank \sF < \rank \sE$, there are two
cases:
\begin{itemize}
\item $[\sF] \not ∈ S$ and it follows from $α$-stability of $\sE$ that
  $$
  μ_{(1-ε)· α+ε· β}(\sF) < μ_{(1-ε)· α+ε· β}(\sE) \quad
  \text{for all } ε ∈ [0,1)
  $$
\item $[\sF] ∈ S$ and there exists an index $j$ such that $[\sF] = [\sF_j]$,
  hence
  $$
  μ_{(1-ε)· α+ε· β}(\sF) = μ_{(1-ε)· α+ε· β}(\sF_j) \quad \text{for
    all } ε ∈ [0,1].
  $$
\end{itemize}
In either case, we obtain the following inequality for all $ε ∈ [0,1]$,
$$
μ_{(1-ε)· α+ε· β}(\sF) ≤ \max \Bigl\{ \underbrace{\max_{1 ≤ j ≤ n}
  μ_{(1-ε)· α+ε· β}(\sF_j)}_{=:Φ(ε)},\, μ_{(1-ε)· α+ε· β}(\sE)
\Bigr\}.
$$
The inequality is strict for all $ε$ for which $Φ(ε) < μ_{(1-ε)· α+ε·
  β}(\sE)$.  \Preprint{The setup is depicted in Figure~\vref{fig:toma}.
  \begin{figure}
    \centering
    \footnotesize
    
    \begin{tikzpicture}[scale=0.5]
      \draw (-1,-1) to (16,-1);
      \filldraw (00,-1) circle (0.05) node[below] {$α$};
      \filldraw (15,-1) circle (0.05) node[below] {$β$};
      \filldraw (05,-1) circle (0.05) node[below] {$(1-e)· α + e· β$};
      
      \draw [dotted] (0,3) node[left]{$μ_{α}(\sE)$} to (15,3) node[right]{$μ_{β}(\sE)$};
      \draw [dotted] (0,2) node[left]{$μ_{α}(\sF_1)$} to (15,5) node[right]{$μ_{β}(\sF_1)$};
      \draw [dotted] (0,0) node[left]{$μ_{α}(\sF_2)$} to (15,6) node[right]{$μ_{β}(\sF_2)$};
      \draw [dotted] (0,1) node[left]{$μ_{α}(\sF)$} to (15,2) node[right]{$μ_{β}(\sF)$};
      
      \draw [very thick] (0,2) to (10,4) to (15,6);
    \end{tikzpicture}
    
    \bigskip
    
    {\small Slopes of the sheaves $\sF_j$ and of a sheaf $\sF$ such that $[\sF] \not ∈ S$ with
      respect to movable classes $(1-ε)· α + ε · β$.  The function $Φ$ is
      outlined in bold.}
    
    \caption{Functions discussed in the proof of Theorem~\ref{thm:toma}}
    \label{fig:toma}
  \end{figure}
} Conclude by observing that $Φ$ is continuous as a function of $ε$, and that
$Φ(0) < μ_{α}(\sE)$.  \qed

%
%
\svnid{$Id: 04.tex 214 2015-03-10 12:29:00Z kebekus $}

\section{Tensor products of semistable sheaves}
\label{sec:tensor}
\subversionInfo
\approvals{
  Daniel & yes \\
  Stefan & yes \\
  Thomas & yes
}

In this section we prove that the reflexive tensor product of sheaves which are
α-semistable for some movable class α is again α-semistable.  The following
notation will be used.

\begin{notation}[Reflexive tensor product]
  Given any two coherent sheaves $\sA$, $\sB$ of $\sO_X$-modules on a normal
  variety $X$, write $\sA \boxtimes \sB := (\sA \otimes \sB)^{**}$.  We refer to
  $\sA \boxtimes \sB$ as the \emph{reflexive tensor product} of $\sA$ and $\sB$.
\end{notation}

\begin{thm}[Semistability of tensor products]\label{thm:mainthm}
  Let $X$ be a normal $\bQ$-factorial projective variety, and let $α ∈ \Mov(X)$
  be any movable class.  If $\sF$ and $\sG$ are torsion-free, positive-rank,
  coherent sheaves on $X$, then the following holds.
  \begin{enumerate}
  \item\label{il:asx} $μ_{α}^{\max}\bigl( \sF \boxtimes \sG \bigr) =
    μ_{α}^{\max}(\sF) + μ_{α}^{\max}(\sG)$.
  \item\label{il:bsx} If $\sF$ and $\sG$ are α-semistable, then $\sF \boxtimes
    \sG$ is likewise α-semistable.
  \end{enumerate}
\end{thm}
 
\begin{cor}[Semistability of reflexive symmetric products]\label{cor:symm}
  In the setting of Theorem~\ref{thm:mainthm}, if $\sF$ is α-semistable and $m
  ∈ \bN^+$ any number, then any reflexive symmetric product $\bigl( \Sym^m \sF
  \bigr)^{**}$ is again α-semistable.
\end{cor}
\begin{proof}
  Observe that $\bigl( \Sym^m \sF \bigr)^{**}$ is a direct summand of
  $\sF^{\boxtimes m}$.
\end{proof}

\subsection{Preparation for the proof: reduction to semistability of tensor products of stable sheaves}
\label{ssec:prrrp}
\approvals{
  Daniel & yes \\
  Stefan & yes \\
  Thomas & yes
}

We show in this section that to prove Theorem~\ref{thm:mainthm}, it suffices
to show that the product of two stable sheaves is semistable.  The
following is the main result of the present Section~\ref{ssec:prrrp}.

\begin{prop}\label{prop:prrrp}
  Let $X$ be a normal, $\bQ$-factorial, projective variety and let $α ∈ \Mov(X)$
  be any movable class.  Suppose that the reflexive tensor product of any two
  torsion-free, α-stable sheaves is α-semistable.  Then, the following holds for
  any two torsion-free sheaves $\sF$ and $\sG$ on $X$.
  \begin{enumerate}
  \item\label{il:A} $μ_{α}^{\max} (\sF \boxtimes \sG ) = μ_{α}^{\max} (\sF) +
    μ_{α}^{\max} (\sG)$.
  \item\label{il:B} If $\sF$ and $\sG$ are α-semistable, then $\sF \boxtimes
    \sG$ is again α-semistable.
  \end{enumerate}
\end{prop}

We prove Proposition~\ref{prop:prrrp} in the remainder of
Section~\ref{ssec:prrrp}.  The proof is subdivided into five steps.

\subsubsection*{Step 1: Setup}
\approvals{
  Daniel & yes \\
  Stefan & yes \\
  Thomas & yes
}

Since numerical classes and slopes are unaffected when modifying $\sF$ and $\sG$
along a subset of codimension at least two, we are free to replace these sheaves
by their double duals, and assume henceforth that the following holds.

\begin{asswlog}
  The sheaves $\sF$ and $\sG$ are reflexive.
\end{asswlog}
\CounterStep

Combining the Harder-Narasimhan filtration and a Jordan-Hölder-filtration as in
Remark~\ref{rem:RHNF}, choose a filtration of $\sF$, say $0 = \sF_0 \subsetneq
\sF_1 \subsetneq \cdots \subsetneq \sF_k = \sF$ such that the following holds.
\begin{enumerate}
\item\label{il:1} The quotients $\sQ_{i+1} := \sF_{i+1}/\sF_i$ are torsion-free
  and α-stable for all $i$.
\item\label{il:2} The sequence of ranks, $(\rank \sF_i)_{0 ≤ i ≤ k}$, is
  strictly increasing.
\item\label{il:3} The sequence of slopes, $\bigl( μ_{α}(\sQ_i) \bigr)_{1 ≤ i ≤
    k}$, is decreasing.
\end{enumerate}
Taking reflexive tensor products with $\sG$, we obtain a filtration of $\sF
\boxtimes \sG$,
\begin{equation}\label{eq:FTP}
  0 = \sF_0 \boxtimes \sG \subsetneq \sF_1 \boxtimes \sG \subsetneq \cdots \subsetneq \sF_k \boxtimes \sG = \sF \boxtimes \sG,
\end{equation}
where each term $\sF_i \boxtimes \sG$ is saturated in $\sF_{i+1} \boxtimes \sG$.
Consequently there exists a big, open subset $X° ⊆ X$ where all sheaves $\sF$,
$\sG$, $\sF_i$ and $\sF_i \boxtimes \sG$, as well as all quotients $\sQ_i$ and
$(\sF_{i+1}\boxtimes \sG)/(\sF_i\boxtimes \sG)$ are locally free.  Since two
reflexive sheaves are isomorphic if and only if they agree over $X°$, the
quotients given by the filtration~\eqref{eq:FTP} can be identified as follows,
\begin{equation}\label{eq:XP2}
  \left( \factor{\sF_{i+1}\boxtimes \sG}{\sF_i\boxtimes \sG} \right)^{**} = \sQ_{i+1} \boxtimes \sG.
\end{equation}
The slopes of these sheaves are computed as follows.
\begin{equation}\label{eq:fx3}
  \begin{aligned}
    μ_{α}(\sQ_{i+1} \boxtimes \sG) & = μ_{α}(\sQ_{i+1}) + μ_{α}(\sG) && \text{Slope of product}\\
    & ≤ μ_{α}(\sQ_1) + μ_{α}(\sG) && \text{by~\ref{il:3}} \\
    & ≤ μ_{α}^{\max}(\sF) + μ_{α}^{\max}(\sG)
  \end{aligned}
\end{equation}

\subsubsection*{Step 2: Proof of Claim~\ref*{il:A} in case where $\sF$ or $\sG$ are α-stable}
\approvals{
  Daniel & yes \\
  Stefan & yes \\
  Thomas & yes
}

The roles of $\sF$ and $\sG$ being symmetric, consider the case where $\sG$ is
$α$-stable.  By assumption, the quotient sheaves $\sQ_i \boxtimes \sG$ will thus
be α-semistable.  Inequality~\eqref{eq:fx3} therefore implies that any morphism
$\sA → \sQ_i \boxtimes \sG$ will be zero, if $\sA$ is any semistable subsheaf of
slope $μ_{α}(\sA) > μ_{α}^{\max}(\sF) + μ_{α}^{\max}(\sG)$.  It follows that any
morphism $\sA → \sF \boxtimes \sG$ will be zero, and that the α-slope of any
semistable subsheaf $\sB ⊆ \sF \boxtimes \sG$ is bounded.  In other words, we
have
$$
μ_{α}(\sB) ≤ μ^{\max}_{α}(\sF)+μ^{\max}_{α}(\sG).
$$
On the other hand, $\sF \boxtimes \sG$ does contain the semistable subsheaf
$\sF_1 \boxtimes \sG = \sQ_1 \boxtimes \sG$, whose slope equals
$μ^{\max}_{α}(\sF) + μ_{α}(\sG) = μ^{\max}_{α}(\sF) + μ^{\max}_{α}(\sG)$.  This
proves Claim~\ref{il:A} in case where one of the factors is α-stable.

\subsubsection*{Step 3: Proof of Claim~\ref*{il:A} in general}
\approvals{
  Daniel & yes \\
  Stefan & yes \\
  Thomas & yes
}

Recalling from \ref{il:1} that quotient sheaves $\sQ_{i+1}$ are α-stable, we
have seen in Step~2 that
$$
μ^{\max}_{α} \bigl(\sQ_{i+1} \boxtimes \sG \bigr) = μ^{\max}_{α} \bigl(\sQ_{i+1}
\bigr) + μ^{\max}_{α} \bigl(\sG \bigr) ≤μ^{\max}_{α} \bigl(\sF \bigr) +
μ^{\max}_{α} \bigl(\sG \bigr).
$$
As in Step~2, this implies that the α-slope of any subsheaf $\sB ⊆ \sF \boxtimes
\sG$ is bounded by $μ^{\max}_{α}(\sF)+μ^{\max}_{α}(\sG)$.  On the other hand,
$\sF \boxtimes \sG$ does contain the reflexive product of the maximally
destabilising subsheaves.  The slope of this product equals $μ^{\max}_{α}(\sF) +
μ^{\max}_{α}(\sG)$.  This proves Claim~\ref{il:A} in general.

\subsubsection*{Step 4: Proof of Claim~\ref*{il:B}}
\approvals{
  Daniel & yes \\
  Stefan & yes \\
  Thomas & yes
}

Now assume that $\sF$ and $\sG$ are both α-semistable.  The slope of the product
is computed as follows,
\begin{equation}\label{eq:rrt2}
  \begin{aligned}
    μ^{\max}_{α}(\sF \boxtimes \sG) & = μ^{\max}_{α}(\sF)+μ^{\max}_{α}(\sG) && \text{Claim~\ref{il:A}} \\
    & = μ_{α}(\sF)+μ_{α}(\sG) && \text{Semistability of $\sF$ and $\sG$} \\
    & = μ_{α}(\sF \boxtimes \sG),
  \end{aligned}
\end{equation}
proving its semistability.  This proves Claim~\ref{il:B} and finishes the proof
of Proposition~\ref{prop:prrrp}.  \qed

\subsection{Reduction to a resolution}
\label{ssec:x1}
\approvals{
  Daniel & yes \\
  Stefan & yes \\
  Thomas & yes
}

We end the preparations for the proof of Theorem~\ref{thm:mainthm} with the
following lemma.  Both its claims follow immediately from
Proposition~\ref{prop:invariance}.

\begin{lem}\label{lem:p2}
  In the setting of Theorem~\ref{thm:mainthm}, let $π: \wtilde X → X$ be any
  resolution of singularities.  Then the following holds.
  \begin{enumerate}
  \item The sheaf $\sF$ is α-stable if and only if $π^{[*]}\sF$ is
    $(π^*α)$-stable.
  \item The sheaf $\sF \boxtimes \sG$ is α-semistable if and only if $π^{[*]}
    \sF \boxtimes π^{[*]} \sG$ is $(π^*α)$-semistable.  \qed
  \end{enumerate}
\end{lem}

\subsection{Proof of Theorem~\ref{thm:mainthm} }
\label{ssec:x2}
\approvals{
  Daniel & yes\\
  Stefan & yes \\
  Thomas & yes
}

We have seen in Proposition~\ref{prop:prrrp} that to prove
Theorem~\ref{thm:mainthm}, it suffices to show that the reflexive tensor product
of any two α-stable sheaves is α-semistable.  So, let $\sF$ and $\sG$ be any two
α-stable, torsion-free, coherent sheaves on $X$.

Combining a resolution of singularities with a classical result of Rossi,
\cite{Rossi68}, we find a smooth, projective variety $\wtilde X$ and a
birational morphism $π: \wtilde X → X$ such that $π^{[*]} \sF$ and $π^{[*]} \sG$
are both locally free.  By Lemma~\ref{lem:p2}, it suffices to establish
$(π^*α)$-semistability of $π^{[*]} \sF \otimes π^{[*]} \sG$.  To simplify
notation we replace $X$ by $\wtilde X$ and assume without loss of generality
that the following holds.

\begin{asswlog}\label{awlog:1}
  The variety $X$ is smooth.  The sheaves $\sF$ and $\sG$ are locally free.
\end{asswlog}

Under these assumptions, Theorem~\ref{thm:mainthm} has been shown by Toma in
\cite[Prop.~6.1]{CP11} if the class α is rational and big.  To show
Theorem~\ref{thm:mainthm} in the general case, let $\sA ⊂ \sF \otimes \sG$ be
any proper subsheaf.  Recall from Remark~\ref{rem:toma} that we can find a
sequence of big, rational classes $β_i ∈ \Stab(\sF) ∩ \Stab(\sG)$ with $\lim β_i
= α$.  Toma's result \cite[Prop.~6.1]{CP11} applies to show that $\sF \otimes
\sG$ is stable with respect to the $β_i$, and hence
$$
μ_{β_i}(\sA) ≤ μ_{β_i}(\sF \otimes \sG) \quad \text{for all } i ∈ \bN.
$$
Using continuity of intersection numbers, we may pass to the limit and obtain
the desired inequality for $\sA$, thus finishing the proof of
Theorem~\ref{thm:mainthm}.  \qed

%
%
\svnid{$Id: 05.tex 214 2015-03-10 12:29:00Z kebekus $}

\section{Bogomolov-Gieseker inequalities}
\label{sec:BGI}
\subversionInfo
\approvals{
  Daniel & yes \\
  Stefan & yes \\
  Thomas & yes
}

We first prove a Bogomolov-Gieseker inequality for locally free sheaves on a
smooth projective surface which are semistable with respect to movable classes.
This is a minor generalisation of \cite[4.7]{Miyaoka87}, which considers
rational classes only.

\begin{thm}[Bogomolov-Gieseker Inequality]\label{thm:BGI}
  Let $X$ be a smooth, projective surface, let $\sE$ be a torsion-free sheaf of
  $\sO_X$-modules, of rank $r$.  If $\sE$ is semistable with respect to a
  movable curve class $α ∈ \Mov(X) \setminus \{0\}$, then
  \begin{equation}\label{eq:BGI}
    Δ(\sE) := 2r · c_2(\sE) - (r-1) · c_1²(\sE) ≥ 0
  \end{equation}
  If $\sE$ is not locally free, then Inequality~\eqref{eq:BGI} is strict.
\end{thm}

\begin{rem}
  If the sheaf $\sE$ in Theorem \ref{thm:BGI} is even $α$-stable, then this
  theorem can be deduced from \cite[4.7]{Miyaoka87}, using the results of
  Section \ref{sec:openness}.
\end{rem}

\subsection{Preparation for the proof of Theorem~\ref*{thm:BGI}}
\approvals{
  Daniel & yes \\
  Stefan & yes \\
  Thomas & yes
}

If $X$ is any normal projective variety, if $A$ is any Cartier divisor on $X$
and $i ∈ \bN$ any number, recall that the function $m \mapsto h^i \bigl( X,\,
\sO_X(m· A) \bigr)$ grows asymptotically at most like $m^{\dim X}$.  To
prepare for the proof of Theorem~\ref{thm:BGI}, we recall the following standard
generalization, cf.~\cite[Part I, Lecture III, Prop.~3.3]{MP97}.

\begin{lem}\label{lem:growth}
  Let $X$ be any projective variety, let $A$, $B$, $C ∈ \operatorname{Div}(X)$
  be any three Cartier-divisors and $i ∈ \bN$ be any number.  If $h^i \bigl( X,
  \sO_X(m· A+B) \bigr)$ grows asymptotically at least like $m^{\dim X}$,
  then $h^i \bigl( X, \sO_X(m· A+B+C) \bigr)$ has the same asymptotic growth
  rate,
  \begin{equation}\label{eq:lem:growth}
    h^i \bigl( X, \sO_X(m· A+B) \bigr) \sim
    h^i \bigl( X, \sO_X(m· A+B+C) \bigr).
  \end{equation}
  In particular, taking $C = -B$ it follows that
  $$
  m^{\dim X} \lesssim h^i \bigl( X, \sO_X(m· A+B) \bigr) \sim h^i \bigl( X,
  \sO_X(m· A) \bigr) \lesssim m^{\dim X},
  $$
  so that equality of growth rates holds.  \Publication{\qed}
\end{lem}
\Preprint{
\begin{proof}
  We prove Equation~\eqref{eq:lem:growth} using induction on the dimension of
  $X$.  If $\dim X = 0$, there is nothing to show.  Let us therefore assume that
  $\dim X > 0$, and that the claim was shown for all varieties of smaller
  dimension.  Now, if $P ⊂ X$ is any prime Cartier divisor, consider the ideal
  sheaf sequence
  $$
  0 → \sO_X(mA +B) →\sO_X(mA +B+P) → \sO_X(mA +B+P)|_P → 0.
  $$
  The following excerpt of the long exact cohomology sequence,
  \begin{multline*}
    \overbrace{H^{i-1}\bigl( P, \sO_X(mA +B+P)|_P)}^{\lesssim m^{\dim X-1}} → H^i\bigl(\sO_X(mA +B) \bigr) → \\
    → H^i\bigl(\sO_X(mA +B+P) \bigr) → \underbrace{H^i\bigl( P, \sO_X(mA
      +B+P)|_P)}_{\lesssim m^{\dim X-1}},
  \end{multline*}
  will then show that we have equality of asymptotic growth rates,
  $$
  h^i \bigl( X, \sO_X(m· A+B) \bigr) \sim h^i \bigl( X, \sO_X(m· A+B+P)
  \bigr).
  $$
  The same line of argument, using the sequence
  $$
  0 → \sO_X(mA +B-P) →\sO_X(mA +B) → \sO_X(mA +B)|_P → 0,
  $$
  yields an analogous equality,
  $$
  h^i \bigl( X, \sO_X(m· A+B) \bigr) \sim h^i \bigl( X, \sO_X(m· A+B-P) \bigr).
  $$
  In essence, we have shown that the asymptotic growth rate does not change when
  adding or subtracting prime divisors.  Write $C$ as a difference of two very
  ample prime divisors to conclude.
\end{proof}}

\subsection{Proof of Theorem~\ref*{thm:BGI}}
\approvals{
  Daniel & yes \\
  Stefan & yes \\
  Thomas & yes
}

The proof roughly follows the line of argument given in \cite[4.3]{Miyaoka87} or
\cite[Part I, Lecture III, 3.9]{MP97}.

\subsubsection*{Step 1: Setup}
\approvals{
  Daniel & yes \\
  Stefan & yes \\
  Thomas & yes
}

We argue by contradiction.  To be precise, we assume that the following holds.

\begin{assumption}\label{ass:ertz}
  There exists a smooth, projective surface $X$, a movable class $α ∈ \Mov(X)
  \setminus \{0\}$ and an α-semistable, torsion-free sheaf $\sE$ such that
  $Δ(\sE) < 0$.
\end{assumption}

These assumptions can be simplified.  Recall from \cite[Sect.~3.4]{MR2665168}
that
$$
Δ(\sE) ≥ Δ(\sE^{**}) ≥ Δ( \sE^* \otimes \sE^{**}),
$$
and that the first inequality is strict if $\sE$ is not locally free.  Since
$\sE^* \otimes \sE^{**}$ is likewise α-semistable by Theorem~\ref{thm:mainthm},
we are free to replace $\sE$ by $\sE^* \otimes \sE^{**}$ throughout the
argument.  Recalling that reflexive sheaves are locally free in codimension two,
this amounts to assuming the following.

\begin{asswlog}\label{asswlog:aswer}
  The sheaf $\sE$ is locally free and has trivial determinant, $\det \sE \cong
  \sO_X$.
\end{asswlog}

\subsubsection*{Step 2: Notation}
\approvals{
  Daniel & yes \\
  Stefan & yes \\
  Thomas & yes
}

Choose a Cartier divisor $H ∈ \operatorname{Div}(X)$ such that both $H$ and
$H-K_X$ are ample.  Recalling that α is movable and therefore $α² ≥ 0$, it
follows from the Hodge index theorem that $H · α > 0$.  The same holds for the
intersection with $H-K_X$, and we obtain the following inequalities that we note
for future reference,
\begin{equation}\label{eq:HIT}
  μ_{α} \bigl( \sO_X(H) \bigr) > 0 \quad\text{and}\quad μ_{α} \bigl( \sO_X(H-K_X) \bigr) > 0.
\end{equation}

Set $Y := \bP(\sE)$, and denote the bundle map by $π : Y → X$.  Choose a divisor
$\mathbb 1 ∈ \operatorname{Div}(Y)$ such that $\sO_Y(\mathbb 1) \cong
\sO_{\bP(\sE)}(1)$.

\subsubsection*{Step 3: Conclusion}
\approvals{
  Daniel & yes \\
  Stefan & yes \\
  Thomas & yes
}

Recall from the computations of \cite[Part I, Lecture III, 3.6]{MP97} that the
assumption $Δ(\sE) < 0$ made in \ref{ass:ertz} implies that $\mathbb 1^{\dim Y}
> 0$.  By Riemann-Roch,
\begin{equation}\label{eq:B1}
  m^{\dim Y} \sim χ \bigl( \sO_Y(m · \mathbb 1) \bigr) = \sum_{i=0}^{\dim Y} (-1)^i · h^i\bigl( Y,\, \sO_Y(m · \mathbb 1) \bigr).
\end{equation}
The Leray spectral sequence implies that $h^i\bigl( Y,\, \sO_Y(m · \mathbb
1) \bigr) = h^i\bigl( X,\, π_* \sO_Y(m · \mathbb 1) \bigr)$, which vanishes
for $i ≥ 3$.  Equation~\eqref{eq:B1} therefore yields
\begin{equation}\label{eq:B2}
  h^0\bigl( Y,\, \sO_Y(m · \mathbb 1) \bigr)
  \sim m^{\dim Y} \quad\text{or}\quad h²\bigl( Y,\, \sO_Y(m ·
  \mathbb 1) \bigr) \sim m^{\dim Y}.
\end{equation}
The subsequent Steps 5 and 6 will show that neither of these two possibilities
is realised in our setup finishing the proof.

\subsubsection*{Step 5: Growth rate of $h^0(\cdots)$}
\approvals{
  Daniel & yes \\
  Stefan & yes \\
  Thomas & yes
}

Assume that $h^0\bigl( Y,\, \sO_Y(m · \mathbb 1) \bigr)$ grows
asymptotically like $m^{\dim Y}$.  Combining the results obtains so far, this
implies the following.
\begin{align*}
  m^{\dim Y} & \sim h^0\bigl( Y,\, \sO_Y(m · \mathbb 1 - π^*H ) \bigr) && \text{Lemma~\ref{lem:growth}}\\
  & = h^0\bigl( X,\, π_* \sO_Y(m) \otimes \sO_X(-H) \bigr) && \\
  & = \dim_{\bC} \Hom_X \bigl( \sO_X(H),\, \Sym^{m} \sE \bigr).
\end{align*}
This is absurd.  In fact, we have seen in \eqref{eq:HIT} that the α-slope of the
invertible sheaf $\sO_X(H)$ is positive.  On the other hand, recall from
Assumption~\ref{asswlog:aswer} that $\Sym^m \sE$ has vanishing first Chern
class.  It is thus α-semistable, with $μ_{α}(\Sym^m \sE) = 0$, by
Corollary~\ref{cor:symm}.  As we have seen in Corollary~\ref{cor:msss}, any
morphism $\sO_X(H) → \Sym^m \sE$ is therefore zero.  It follows that $h^0\bigl(
Y,\, \sO_Y(m · \mathbb 1) \bigr)$ cannot grow like $m^{\dim Y}$.

\subsubsection*{Step 6: Growth rate of $h²(\cdots)$}
\approvals{
  Daniel & yes \\
  Stefan & yes \\
  Thomas & yes
}

We repeat the argument of Step~5 with minor variations.  Assuming that $h²
\bigl( Y,\, \sO_Y(m · \mathbb 1) \bigr)$ grows like $m^{\dim Y}$, we obtain,
using again Lemma~\ref{lem:growth}
\begin{align*}
  m^{\dim Y} & \sim h² \bigl( Y,\, \sO_Y(m · \mathbb 1 + π^*H ) \bigr) && \text{Lemma~\ref{lem:growth}}\\
  & = h²\bigl( X,\, π_* \sO_Y(m) \otimes \sO_X(H) \bigr) && \text{Leray spectral sequence}\\
  & = h^0\bigl( X,\, \Sym^{m} \sE^* \otimes \sO_X(-(H-K_X)) \bigr) && \text{Serre duality}\\
  & = \dim_{\bC} \Hom_X \bigl( \sO_X(H-K_X),\, \Sym^{m} \sE^* \bigr).
\end{align*}
As before, recall from \eqref{eq:HIT} that the α-slope of the invertible sheaf
$\sO_X(H-K_X)$ is positive and observe that $\Sym^m \sE^*$ is α-semistable with
slope $μ_{α}(\Sym^m \sE^*) = 0$.  It follows that $h² \bigl( Y,\, \sO_Y(m ·
\mathbb 1) \bigr)$ cannot grow like $m^{\dim Y}$.  \qed

%
%
\svnid{$Id: 06.tex 219 2015-03-12 07:34:05Z kebekus $}

\section{Flatness criteria}
\label{sec:flatness}
\subversionInfo

\subsection{Flatness}
\approvals{
  Daniel & yes \\
  Stefan & yes \\
  Thomas & yes
}

In this section we discuss various flatness results.  The relevant notion is the
following.

\begin{defn}[Flat and projectively flat sheaves]\label{defn:flat}
  If $Y$ is any algebraic variety, and $\sG$ is any locally free, analytic sheaf
  of rank $r$ on the underlying complex space $Y^{an}$, we call $\sG$
  \emph{flat} if it is defined by a representation
  $$
  π_1(Y^{an}) → \operatorname{Gl}(r,\bC)
  $$
  of the topological fundamental group $π_1(Y^{an})$.  A locally free, algebraic
  sheaf on $Y$ is called flat if and only if the associated analytic sheaf is
  flat.
  
  In a similar vein, call $\sG$ \emph{projectively flat} if $\bP(\sG)$ is given by a
  representation $π_1(Y^{an}) → \operatorname{PGl}(r,\bC)$.
\end{defn}

We refer the reader to \cite{GKP13} for a discussion of flat sheaves on
algebraic varieties.  Projective flatness is discussed in \cite{MR3030068} and
in the references quoted there.

\subsubsection{Flatness for sheaves on surfaces}
\approvals{
  Daniel & yes \\
  Stefan & yes \\
  Thomas & yes
}

Simpson proved in \cite{MR1179076} an important flatness criterion for
semistable locally free sheaves.  In fact, let $X$ be a projective manifold $X$
of dimension $n$ and let $\sE$ be a locally free sheaf that is $H$-semistable
for some ample divisor $H$.  Suppose further that
$$
c_1(\sE) · H^{n-1} = (c_1(\sE)² - c_2(\sE)) · H^{n-2} = 0.
$$
Then, $\sE$ is flat.  We first generalise this in case $\dim X = 2$, replacing
$H$ by a movable class on the surface.

\begin{thm}[Criterion for flatness]\label{thm:flat:1}
  Let $X$ be a smooth, projective surface and $α ∈ \Mov(X)$ a movable class.
  Let $\sE$ be any torsion-free coherent sheaf on $X$ with
  \begin{equation}\label{eq:simpson}
    c_1(\sE) · α = c_1(\sE)² - c_2(\sE) = 0.
  \end{equation}
  If $\sE$ is α-semistable and $α² > 0$, then $\sE$ is a locally free, flat
  sheaf.
\end{thm}

\begin{rem}
  The assumptions made in Theorem~\ref{thm:flat:1} are necessary.  For an
  example where $\sE$ is α-semistable, where $α² = 0$ and where $\sE$ is not
  flat, let $π : X → Y$ be any ruled surface, let $F$ be any fibre and set $α :=
  [F]$ and $\sE := \sO_X(F) \oplus \sO_X(-F)$
\end{rem}

\begin{thm}[Criterion for projective flatness]\label{thm:flat:2a}
  Let $X$ be a smooth, projective surface and $α ∈ \Mov(X) \setminus \{ 0\}$ be
  a movable class.  Let $\sE$ be any α-stable, torsion-free coherent sheaf on
  $X$ of rank $r$.  If equality holds in the Bogomolov-Gieseker inequality,
  $$
  2r · c_2(\sE) = (r-1) · c_1(\sE)²,
  $$
  then $\sE$ is a projectively flat, locally free sheaf that is semistable with
  respect to every $β ∈ \Mov(X)$.  If we assume in addition that $c_1(\sE) · α =
  c_1(\sE)² = 0$, i.e., we assume that \eqref{eq:simpson} holds, then either
  $\sE$ or $\sE^*$ is nef.
\end{thm}

\subsubsection{Flatness in higher dimensions}
\approvals{
  Daniel & yes \\
  Stefan & yes \\
  Thomas & yes
}

We use the flatness criterion of Theorem~\ref{thm:flat:1} to generalise
Theorem~1.19 from \cite{GKP13}.

\begin{thm}[Flatness criterion]\label{thm:flat:2b}
  Let $X$ be a normal, projective, $\bQ$-factorial variety of dimension $n$ with
  only canonical singularities.  Let $\sE$ be a reflexive sheaf on $X$ and $H ∈
  \Div(X)$ an ample divisor.  Suppose that $\sE$ is $H$-semistable and that
  there exists a desingularisation $π: \wtilde X → X$ such that the following
  two equalities hold,
  \begin{align}
    \label{eq:vanish1} 0 & = c_1(\sE) · H^{n-1} \\
    \label{eq:vanish2} 0 & = c_1 \bigl( π^{[*]} \sE \bigr)² · (π^* H)^{n-2} - c_2 \bigl( π^{[*]} \sE \bigr) · (π^*H)^{n-2}.
  \end{align}
  Then, there exists a quasi-étale morphism $γ: \what X → X$ such that $γ^{[*]}
  \sE$ is a locally free, flat sheaf on $\what X$.
\end{thm}

\begin{rem}
  Notice that condition \eqref{eq:vanish2} does not depend on the choice of $π$,
  see \cite[4.4]{GKP13} for the argument.  It is essential that two
  desingularizations can be dominated by a third one.
\end{rem}

\begin{rem}
  Theorem~\ref{thm:flat:2b} holds if $X$ has only klt singularities provided we
  knew that $π_1(X) \simeq π_1(X°)$, where $X°$ is the locus of quotient
  singularities.
\end{rem}
  
\begin{rem}
  One might ask whether Theorem~\ref{thm:flat:2b} holds in greater generality,
  that is, for a larger class of singular spaces.  What we actually needed is
  the following.
  \begin{enumerate}
  \item A general complete intersection surface $S ⊂ X$ has only rational
    singularities.
  \item The following Lefschetz-type statement holds true: there exists a closed
    set $A ⊂ X$ of codimension at least three such that $π_1(S) \simeq π_1(X
    \setminus A)$.
  \item There exists a quasi-étale cover $\wtilde X → X$ such that
    $\what{π}_1(\wtilde{X}_{\reg}) \simeq \what{π}_1(\wtilde{X})$.
 \end{enumerate}
\end{rem}

\begin{cor}\label{cor:polystable}
  Let $X$ be a normal $\bQ$-factorial projective variety with only canonical
  singularities.  Let $H \ne 0$ be any nef divisor on $X$ and $\sE = \bigoplus
  \sE_i$ a reflexive sheaf on $X$ whose direct summands are $H$-stable.  If
  $c_1(\sE_i) = 0$ for all $i$ and if there exists a desingularisation $π:
  \widetilde X → X$ such that $π^{[*]}\sE$ is locally free and $c_2 \bigl(
  π^{[*]}\sE \bigr) = 0$, then there exists a quasi-étale morphism $γ: \what{X}
  → X$ such that $γ^{[*]}\sE$ is a flat, locally free sheaf on $\what{X}$.
\end{cor}
\begin{proof}
  If $H_0$ is any ample divisor, it will follow from Theorem~\ref{thm:toma} that
  all direct summands $\sE_i$ are stable with respect to $(H + ε · H_0)$, for
  all sufficiently small positive $ε$.  More is true.  Since $c_1(\sE_i) = 0$
  the bundle $\sE$ is polystable with respect to $(H + ε · H_0)$, and in
  particular semistable.  Now apply Theorem~\ref{thm:flat:2b}.
\end{proof}

\subsection{Preparations}
\approvals{
  Daniel & yes \\
  Stefan & yes \\
  Thomas & yes
}

The proof of Theorem~\ref{thm:flat:1} makes use of the following lemmas, which
might be of independent interest.

\begin{lem}\label{lem:61-1}
  Let $X$ be a smooth, projective surface and $Δ ∈ N¹(X)_{\bQ}$ a rational
  divisor class.  Then, either there exists an integral, ample divisor $H$ such
  that $H · Δ = 0$, or either $Δ$ or $-Δ$ is pseudo-effective.
\end{lem}
\begin{proof}
  Since $X$ is a surface, divisors and curves coincide, and the interior of the
  movable cone equals the ample cone.  Since $Δ$ is a rational class by
  assumption, the hyperplane
  $$
  Δ^\perp := \{ α ∈ N_1(X)_{\bR} \:|\: α · Δ = 0\}
  $$
  is likewise rational, and rational points are dense there.  If $Δ^\perp$
  intersects the interior of the movable cone at all, the intersection will thus
  contain a rational point.  In other words, there will be integral ample
  divisors $H$ with $H · Δ = 0$.

  Assuming that no such divisor $H$ exists therefore amounts to assuming that
  the function $\bullet · Δ$ does not have any zeros in $\Mov(X)°$ and is either
  strictly positive or strictly negative there.  As seen in
  Remark~\vref{rem:BDPP}, this means that either $Δ$ or $-Δ$ is
  pseudo-effective.
\end{proof}

\begin{lem}\label{lem:61-2}
  Let $X$ be a smooth, projective surface and $Δ ∈ N¹(X)_{\bQ}$ a rational,
  pseudo-effective divisor class with $Δ² = 0$.  Assume that there exists a
  movable class $α ∈ \Mov(X) \setminus \{0\}$ such that $α· Δ = 0$.  Then, $Δ$
  is nef.
\end{lem}
\begin{proof}
  Consider the Zariski decomposition, $Δ = [P] + [N]$, where $P$ and $N$ are
  $\bQ$-divisors, where $P$ is nef, $N = \sum a_i N_i$ is effective, $N² < 0$
  unless $N = 0$, and finally $P · N_i = 0$ for all $i$.

  Observe that the movable (=nef) class $α$ intersects $P$ and $N$
  non-negatively.  The assumption that $α · Δ = 0$ therefore gives that $α · P =
  α · N = 0$.  Since $α \ne 0$ and $α² ≥ 0$, the Hodge index theorem thus
  implies that the non-negative number $P²$ is actually zero.  Since $0 = Δ² =
  P² + N²$, we obtain that $N²=0$, and hence that $N=0$.  It follows that $Δ =
  [P]$ is nef.
\end{proof}

\subsection{Proof of Theorem~\ref*{thm:flat:1}}
\approvals{
  Daniel & yes \\
  Stefan & yes \\
  Thomas & yes
}
\CounterStep

Using the assumption that $α²>0$, the Hodge index theorem asserts that
$c_1(\sE)² ≤ 0$.  Thus, the assumption $c_1(\sE)² = c_2(\sE)$, combined with the
Bogomolov-Gieseker inequality of Theorem~\ref{thm:BGI} yields $c_1(\sE)² =
c_2(\sE) = 0$, and taking into account the further assumption made in
\eqref{eq:simpson} the Hodge index theorem hence asserts that $c_1(\sE) = 0$.
In particular, equality holds in the Bogomolov-Gieseker inequality of
Theorem~\ref{thm:BGI}, and $\sE$ is therefore locally free.

\subsubsection*{Step 1: Proof in case where $\sE$ is stable}
\approvals{
  Daniel & yes \\
  Stefan & yes \\
  Thomas & yes
}

If $\sE$ is $α$-stable, Theorem~\ref{thm:toma} and Remark~\ref{rem:toma} allow
to find a rational point $h' ∈ [\sE]^\perp ∩ \Mov(X)° ∩ \Stab(\sE)$.  If $H'$ is
a divisor with class $h'$, $[H'] = h'$, then $H'$ is ample, and it follows from
classical theorems that $\sE$ is flat, e.g.~see \cite[Cor.~3.10]{MR1179076}.  To
apply Simpson's theorem, we view $\sE$ as a Higgs bundle with zero Higgs field
$θ = 0$ and notice that slope semi-stability in our sense implies semi-stability
of Higgs bundles.  \qed

\subsubsection*{Step 2: Proof in general}
\approvals{
  Daniel & yes \\
  Stefan & yes \\
  Thomas & yes
}

Consider a minimal-length Jordan-Hölder filtration of $\sE$, as discussed in
Corollary~\ref{cor:JH},
$$
0 = \sE_0 \subsetneq \sE_1 \subsetneq \cdots \subsetneq \sE_k = \sE.
$$
The proof proceeds by induction on the length of the filtration, denoted by $k$.
If $k = 1$, then $\sE$ is α-stable, and the assertion has been shown in Step~1.
We will therefore assume for the remainder of the proof that $k > 1$, and that
the claim has already been shown for all α-semistable bundles satisfying
Equations~\eqref{eq:simpson} that admit a shorter Jordan-Hölder filtration.

For brevity of notation write $\sF := \sE_{k-1}$ and $\sQ := \sE/\sF$.  We
obtain a sequence
$$
0 → \sF → \sE → \sQ → 0,
$$
where $\sF$ is α-semistable, reflexive and therefore locally free, and where
$\sQ$ is α-stable and torsion-free, and $c_1(\sF) · α = c_1(\sQ) · α = 0$.  As
before, the Hodge index theorem will thus imply the following.
\begin{equation}\label{eq:X2b}
  c_1(\sF)² ≤ 0 \quad\text{and}\quad c_1(\sQ)² ≤ 0, \quad \text{each with equality iff } c_1(·)=0.
\end{equation}
We aim to show that the vanishing \eqref{eq:simpson} holds for $\sF$ and $\sQ$.
Since both $\sF$ and $\sQ$ admit Jordan-Hölder filtrations of length less than
$k$, flatness of $\sF$ and $\sQ$ then follows from the induction hypothesis.
The bundle $\sE$ is then presented as an extension of two flat bundles and is
therefore flat by \cite[3.10]{MR1179076}.

To this end, we use the following relations between the Chern classes of $\sF$
and $\sQ$,
\begin{align}
  0 & = c_1(\sE)² = c_1(\sF)² + c_1(\sQ)² + 2 · c_1(\sF) · c_1(\sQ) \label{eq:ET:1}\\
  0 & = c_2(\sE) = c_2(\sF) + c_2(\sQ) + c_1(\sF) · c_1(\sQ) \label{eq:ET:2}
\end{align}
These equations yield
\begin{align*}
  \textstyle \frac{1}{2}c_1(\sF)² + \frac{1}{2} c_1(\sQ)² & = c_2(\sF) + c_2(\sQ) && \text{\eqref{eq:ET:1} and \eqref{eq:ET:2}} \\
  & ≥ \underbrace{\textstyle \frac{\rank \sF-1}{2\rank \sF}}_{\mathclap{\text{non-neg., }< 1/2}} · \; c_1(\sF)² + \underbrace{\textstyle \frac{\rank \sQ-1}{2\rank \sQ}}_{\mathclap{\text{non-neg., }< 1/2}} · \; c_1(\sQ)² && \text{BGI, Theorem~\ref{thm:BGI}.}
\end{align*}
Combined with \eqref{eq:X2b}, this shows that $c_1(\sF)² = c_1(\sQ)² = 0$.
Applying the Bogomolov-Gieseker Inequality, Theorem~\ref{thm:BGI}, once more to
$\sF$ and $\sQ$, we see that $c_2(\sF) ≥ 0$ and $c_2(\sQ) ≥ 0$, and it follows
from \eqref{eq:ET:2} that both numbers vanish.  The induction hypothesis
therefore applies to show that $\sF$ and $\sQ$ are flat.  As noted above, this
proves flatness of $\sE$.  \qed

\subsection{Proof of Theorem~\ref*{thm:flat:2a}}
\approvals{
  Daniel & yes \\
  Stefan & yes \\
  Thomas & yes
}

Let $\sE$ be any $α$-stable, torsion-free sheaf on $X$ such that equality holds
in the Bogomolov-Gieseker inequality.  By Theorem~\ref{thm:BGI}, the sheaf $\sE$
is then locally free.

\subsubsection*{Step 1: Proof of projective flatness}
\approvals{
  Daniel & yes \\
  Stefan & yes \\
  Thomas & yes
}

We have seen in Theorem~\ref{thm:toma} that $\sE$ is stable also with respect to
a suitable rational class $β$ that is contained in the interior of the movable
cone.  Since $X$ is a surface, the movable cone equals the nef cone, and its
interior consists of ample classes.  Projective flatness is now a consequence of
a classical theorem, see \cite[Thm.~1.1 and Prop.~1.1]{MR3030068} for a
discussion and for further references.  This paper also asserts that $\sE$ is
semistable with respect to any ample class.  Semistability with respect to all
classes in $\Mov(X) \setminus \{ 0\}$ then follows from Remark~\ref{rem:cvx},
since $\Mov(X)$ is the closure of the ample cone.

\subsubsection*{Step 2: Proof of nefness}
\approvals{
  Daniel & yes \\
  Stefan & yes \\
  Thomas & yes
}

Assuming that the Equalities~\eqref{eq:simpson} hold, we aim to prove that $\sE$
or $\sE^*$ are nef, or equivalently, that $\Sym^r(\sE)$ or $\Sym^r(\sE^*)$ are
nef vector bundles, where $r := \rank \sE$, cf.~\cite[Thm.~6.2.12]{Laz04-II}.
Since both $\Sym^r(\sE) \otimes \det \sE^*$ \emph{and} its dual are nef by
projective flatness, \cite[Sect.~1.3]{MR3030068}, it suffices to show that
either $\det \sE$ or $\det \sE^*$ is nef.

To this end, observe that if there exists an integral, ample divisor $H$ such
that $H · c_1(\sE) = 0$, then Theorem~\ref{thm:toma} allows to find a rational
point $h' ∈ c_1(\sE)^\perp ∩ \Mov(X)° ∩ \Stab(\sE)$.  Any divisor $H'$ with
class $h'$ is then ample, and it follows from Simpson's theorem that $\sE$ is
flat, \cite[Cor.~3.10]{MR1179076}.  The class $c_1(\sE)$ will then vanish, and
we are done.

We will thus assume that no such divisor $H$ exists.  One of the divisors $\det
\sE$ or $\det \sE^*$ is then pseudo-effective by Lemma~\ref{lem:61-1}.  Using
the assumption that $c_1(\sE) · α = c_1(\sE)² = 0$, nefness now follows from
Lemma~\ref{lem:61-2}.  \qed

\subsection{Proof of Theorem~\ref*{thm:flat:2b}}
\approvals{
  Daniel & yes \\
  Stefan & yes \\
  Thomas & yes
}

Our subsequent proof of Theorem~\ref{thm:flat:2b} will use the following minimal
generalisation of a result in \cite{GKP13}.  The proof is exactly the same as
the one given there, and therefore omitted.

\begin{lem}[{Iterated Bertini-type theorem for bounded families, cf.~\cite[Cor.~5.6]{GKP13}}]\label{lem:bertstabY}
  Let $X$ be a normal, projective variety of dimension $\dim X ≥ 2$.  Let $\sE$
  be a coherent, reflexive sheaf of $\sO_X$-modules, and let $F$ be a bounded
  family of locally free sheaves.  Given an ample line bundle $\sL ∈ \Pic(X)$, a
  sufficiently increasing sequence $0 \ll m_1 \ll m_2 \ll \cdots \ll m_{k}$ and
  general elements $H_i$ of a basepoint-free linear system contained in
  $|\sL^{\otimes m_i}|$ with associated complete intersection variety $S := H_1
  ∩ \cdots ∩ H_k$, then the following holds for all sheaves $\sF ∈ F$.  The
  sheaf $\sF$ is isomorphic to $\sE$ if and only if $\sF|_S$ is isomorphic to
  $\sE|_S$.  \qed
\end{lem}

\begin{proof}[Proof of Theorem~\ref*{thm:flat:2b}]
  Let $γ: \widehat X → X$ be the finite, quasi-étale cover guaranteed by
  \cite[Thm.~1.13]{GKP13}, which has the property that any locally free, flat
  sheaf defined on $\widehat X_{\reg}^{an}$ extends to a locally free, flat
  sheaf on the whole of $\widehat X$.

  We may assume $H$ to be sufficiently ample in the sense of Mehta-Ramanathan
  and so that Lemma~\ref{lem:bertstabY} may be applied to the linear subsystem
  $γ^*|H|$ of $|γ^*(H)|$ without taking further multiples.  Choose a general
  tuple of hypersurfaces $D_1, …, D_{n-2} ∈ |H|$ and write
  $$
  S := D_1 ∩ \cdots ∩ D_{n-2}.
  $$
  We recall the classical fact, that in codimension two a variety with canonical
  singularities is locally in the Euclidean topology a product of a surface with
  an ADE singularity and of a smooth space of dimension $(\dim X - 2)$, see for
  example~\cite[Prop.~9.3]{GKKP11}.  In particular, the set $Z ⊂ X$ where $X$ is
  \emph{not} locally a complete intersection is small, $\codim_X Z ≥ 3$.  Set
  $X° := X \setminus Z$ and note that $S ⊂ X°$.
  
  Since $S$ is likewise canonical, it is normal with only ADE singularities.  In
  particular, $S$ is $\bQ$-factorial.  By Flenner's version of the
  Mehta-Ramanathan theorem, \cite[Thm.~1.2]{Flenner84}, the restriction $\sE|_S$
  is semistable with respect to the ample class $H|_S$.
  
  Let $\wtilde D_i := π^*(D_i) ∈ | π^*H |$ and set $\wtilde S := \wtilde D_1 ∩
  \cdots ∩ \wtilde D_{n-2}$.  Then $\wtilde S = π^{-1}(S)$ is a smooth surface
  contained in the smooth locus of $\wtilde X$.  For convenience of notation,
  write
  $$
  g := π|_S, \quad \wtilde{\sE} := g^{[*]}(\sE|_S) \quad \text{and} \quad
  \wtilde{H} := g^*(H|_S).
  $$
  We have seen in Proposition~\vref{prop:invariance} that the pull-back sheaf
  $\wtilde{\sE}$ is stable with respect to $\wtilde{H}$.  Equations
  \eqref{eq:vanish1} and \eqref{eq:vanish2} then read as follows,
  $$
  0 = c_1(\wtilde \sE) · \wtilde H \quad \text{and} \quad 0 = c_1(\wtilde
  \sE)² - c_2(\wtilde \sE).
  $$
  For the second equation, we observe that $π^{[*]}(\sE)|_{\wtilde S} $ is
  reflexive and thus equals $\wtilde \sE$.  In particular,
  Theorem~\ref{thm:flat:1} implies that $\wtilde{\sE}$ is flat, hence given by a
  representation $ π_1(\wtilde S) → \Gl(r,\bC)$.  Since $S$ has only canonical
  singularities, $π_1(\wtilde S) \cong π_1(S)$, and the induced representation
  $π_1(S) → \Gl(r,\bC)$ defines a locally free flat sheaf $\sF$ on $S$ such that
  $p^* \sF \simeq \wtilde \sE$.  It follows that the sheaves $\sF$ and $\sE|_S$
  agree outside a finite set of $S$.  Since $\sE|_S$ is reflexive,
  \cite[Proposition~5.2]{GKP13}, the sheaves $\sF$ and $\sE|_S$ are isomorphic.
  Since $X°$ is a locally complete intersection variety, a Lefschetz-type
  theorem of Goresky-MacPherson \cite[II.1.2, Thm.~on p.~153]{GoreskyMacPherson}
  asserts that the inclusion $ι: S → X°$ induces an isomorphism
  $$
  ι_*: π_1(S) → π_1(X°).
  $$
  Thus, there exists a locally free, flat sheaf $\sG$ on $X° $ such that $\sG|_S
  \cong \sE|_S$.  The observation that Zariski-open subsets with complement of
  codimension at least two inside a smooth variety have the same fundamental
  group as the entire variety together with the choice of $\widehat X$ implies
  that the pullback of $\sG$ to $γ^{-1}(X°) ∩ \widehat X_{\reg}$ extends to a
  locally free, flat sheaf $\widehat \sG$ on the whole of $\widehat X$.  Setting
  $\widehat S := γ^{-1}(S) = γ^{*}D_1 ∩ \dots ∩ γ^{*}D_{n-2}$, we note that
  $$
  \widehat \sG |_{\widehat S}\cong (γ|_{\widehat S})^* (\sE|_S) \cong \bigl(γ^{[*]}\sE\bigr)|_{\widehat S}.
  $$
  As the set of all locally free, flat sheaves on $\widehat X$ whose rank is
  equal to $\rank \sE$ is bounded by \cite[Prop.~12.1]{GKP13}, we may apply
  Lemma~\ref{lem:bertstabY} above to conclude that $γ^{[*]}\sE$ is isomorphic to
  $\widehat \sG$, and hence locally free and flat, as claimed.
\end{proof}

%
%
\svnid{$Id: 07-tquot.tex 82 2014-07-21 07:55:34Z peternell $}

\section{Characterisation of torus quotients}
\label{sec:tquotnew}
\subversionInfo
\approvals{
  Daniel & yes \\
  Stefan & yes \\
  Thomas & yes
}

The main result of this section generalises a result for three-dimensional
varieties of Shepherd-Barron and Wilson \cite[Cor.~of Main Thm]{SBW94}, and
eliminates the a priori assumption on the codimension of the singular locus made
in \cite[Thm.~1.16]{GKP13}.

\begin{thm}[Characterisation of torus quotients]\label{thm:tquot}
  Let $X$ be a normal $\bQ$-factorial projective variety of dimension $n$ with
  only canonical singularities and numerically trivial canonical bundle, $K_X
  \equiv 0$.  Assume that there exists a desingularisation $π: \wtilde X → X$
  and an ample divisor $H ∈ \operatorname{Dix}(X)$ such that $c_2(\wtilde X)
  · (π^*H)^{n-2} = 0$.  Then, $X$ is smooth in codimension two, there exists
  an Abelian variety $A$ and a quasi-étale morphism $γ: A → X$.
\end{thm}

Theorem~\ref{thm:tquot} is shown below, \vpageref{pf:thm:tquot}.  It follows
almost immediately from the following two lemmas.

\begin{lem}\label{lem:7-1}
  Let $π : \wtilde{S} → S$ be a desingularisation of a normal, projective,
  $\bQ$-factorial surface $S$ with only rational singularities.  Let $α ∈
  \Mov(S)$ be a movable class with $α² > 0$ and let $\sE$ be a torsion-free,
  $(π^*α)$-semistable, coherent sheaf on $\wtilde S$ such that
  $$
  c_1(\sE) · \bigl( π^*α \bigr) = c_1(\sE)² - c_2(\sE) = 0.
  $$
  Then, $π_* \sE$ is locally free and flat on $S$.
\end{lem}
\begin{proof}
  Write $r := \rank \sE$.  It follows directly from Theorem~\ref{thm:flat:1}
  that $\sE$ is flat, given by a representation $ρ: π_1(\widetilde S) →
  GL_{r}(\bC)$.  If $x ∈ S$ is any singular point with associated fibre $F :=
  π^{-1}(x)$, then $F$ is simply-connected and $\sE|_F$, which is given by a
  representation of $π_1(F)$, therefore trivial.  Consequently, if $\sE'$ is the
  locally free, flat sheaf defined by $ρ ◦ (π_*)^{-1} : π_1(S) → GL_{r}(\bC)$,
  we have $\sE = π^* \sE'$ and $\sE' = π_* \sE$.
\end{proof}

\begin{lem}\label{lem:7-2}
  Let $X$ be a normal, projective variety of dimension $n ≥ 2$ with numerically
  trivial canonical class, $K_X \equiv 0$, that has at most canonical
  singularities.  Let $π : \wtilde X → X$ be any resolution of singularities and
  $H$ an ample divisor on $X$.  If $c_2(\wtilde X) · (π^*H)^{n-2} = 0$, then
  $X$ is smooth in codimension two.
\end{lem}
\begin{proof}
  We may assume without loss of generality that $\sH := \sO_X(H)$ is very ample.
  Choose a general $(n-2)$-tuple of elements $H_i ∈ | \sH |$ and consider the
  associated complete intersection surface $S := ∩_i H_i ⊆ X$, which has at
  worst canonical singularities.  We aim to show that $X$ is smooth near $S$.
  Write
  $$
  \sN_{S/X} := (\sH|_S)^{\oplus n-2}, \quad \wtilde S := π^{-1}(S) \quad\text{and}\quad π_S :=
  π|_{\wtilde S}.
  $$
  Since $X$ has canonical singularities, we may write $K_{\wtilde X} = π^*K_X +
  D \equiv D$, where $D$ is effective and $π$-exceptional.  We obtain the
  following equalities of intersection numbers,
  $$
  \bigl[K_{\wtilde X}|_{\wtilde S} \bigr] · \bigl[\sN^*_{\wtilde S/\wtilde
    X}\bigr] = \bigl[D|_{\wtilde S}\bigr] · \bigl[\sN^*_{\wtilde S/\wtilde
    X}\bigr] = - \bigl[D|_{\wtilde S}\bigr] · \bigl[ π_S^* \sN_{S/X} \bigr] =
  0
  $$
  and hence
  \begin{equation}\label{eq:cFx}
    \bigl[K_{\wtilde S} \bigr] · \bigl[\sN^*_{\wtilde S/\wtilde X} \bigr] =
    \Bigl( \bigl[K_{\wtilde X}|_{\wtilde S} \bigr] + \bigl[\sN_{\wtilde S/\wtilde
      X} \bigr] \Bigr) · \bigl[\sN^*_{\wtilde S/\wtilde X} \bigr] = - \bigl[
    \sN_{\wtilde S/\wtilde X}\bigr]².
  \end{equation}
  The normal bundle sequence for $\wtilde S$ in $\wtilde X$ expresses the
  relevant second Chern class as follows,
  \begin{equation}\label{eq:AC}
    \begin{aligned}
     0= c_2 \bigl( Ω¹_{\wtilde X}|^{\vphantom{1}}_{\wtilde S} \bigr) & = c_2 \bigl( \sN_{\wtilde S/\wtilde X}^* \bigr) + c_2 \bigl( Ω¹_{\wtilde S} \bigr) + \bigl[K_{\wtilde S} \bigr] · \bigl[\sN_{\wtilde S/\wtilde X}^*\bigr] \\
      & = c_2 \bigl( \sN_{\wtilde S/\wtilde X}^* \bigr) + χ_{\operatorname{top}} \bigl(\wtilde S \bigr) - \bigl[\sN_{\wtilde S/\wtilde X}\bigr]² && \quad\text{by \eqref{eq:cFx}.}
    \end{aligned}
  \end{equation}

  We now compare $π$ to a $\bQ$-factorial terminalisation $ρ : \what X → X$ of
  $X$, which exists by \cite[Cor.~1.4.3]{BCHM10}.  As the name suggests, $\what
  X$ is $\bQ$-factorial and has at most terminal singularities.  Moreover, as
  $X$ has canonical singularities, $ρ$ is crepant, and
  \begin{equation}\label{eq:tkx}
    K_{\what X} \: \sim_{\bQ} \: ρ^* K_X \equiv 0.
  \end{equation}
  Write $\what S := ρ^{-1}(S)$ and $ρ_S := ρ|_{\what S}$.  Since varieties with
  terminal singularities are smooth in codimension two, the surface $\what S$ is
  smooth, and entirely contained in the smooth locus of $\what X$.  As above,
  compute
  \begin{equation}\label{eq:DC}
    \begin{aligned}
      c_2 \bigl( Ω¹_{\what X}|^{\vphantom{1}}_{\what S} \bigr) & = c_2 \bigl( \sN_{\what S/\what X}^*\bigr) + c_2 \bigl( Ω¹_{\what S} \bigr) + \bigl[K_{\what S} \bigr] · \bigl[\sN_{\what S/\what X}\bigr] \\
      & = c_2 \bigl( \sN^*_{\what S/\what X} \bigr) + χ_{\operatorname{top}} \bigl(\what S \bigr) - \bigl[\sN_{\what S/\what X}\bigr]².
    \end{aligned}
  \end{equation}
  It follows from adjunction and from the $\bQ$-linear equivalence $K_{\what X}
  \sim_{\bQ} ρ^* K_X$ that $K_{\what S} \sim_{\bQ} (ρ|_{\what S})^* K_S$, so
  that $\what S$ is in fact the minimal resolution of $S$.  Consequently, there
  exists a birational $S$-morphism $β : \wtilde S → \what S$, which is a
  sequence of blowing-down $(-1)$-curves.  This has two consequences.  First,
  there is an inequality of topological Euler characteristics,
  $χ_{\operatorname{top}} \bigl(\what S \thinspace \bigr) ≤
  χ_{\operatorname{top}} \bigl(\wtilde S \bigr)$.  Secondly, we see that
  $$
  \sN_{\wtilde S/\wtilde X} = π_S^* \sN_{S/X} = β^* ρ_S^* \sN_{S/X} = β^* \sN_{\what S/\what X},
  $$
  so that the Chern numbers of the normal bundles agree, $c_2( \sN_{\wtilde
    S/\wtilde X}) = c_2( \sN_{\what S/\what X})$ and $[\sN_{\wtilde S/\wtilde
    X}]² = [\sN_{\what S/\what X}]²$.  Comparing \eqref{eq:AC} and
  \eqref{eq:DC}, we thus see that $c_2 \bigl( Ω¹_{\what
    X}|^{\vphantom{1}}_{\what S} \bigr) ≤ 0$.

  To end the argument, let $\check X → \what X$ be a strong resolution of
  singularities, with induced map $r: \check X → X$, surface $\check S :=
  r^{-1}(S) ⊆ \check X$ and restriction $r_S := r|_{\check S}$.  Since $\check
  X$ and $\what X$ are isomorphic along $\what S$, we have $c_2 \bigl(
  Ω¹_{\check X}|^{\vphantom{1}}_{\check S} \bigr) ≤ 0$, and Miyaoka's Chern
  class inequalities, \cite[Thm.~1.1]{Miyaoka87}, assert that in fact equality
  holds, $c_2 \bigl( Ω¹_{\check X}|^{\vphantom{1}}_{\check S} \bigr) = 0$.

  To sum up, we have seen in \eqref{eq:tkx} that
  $$
  c_1 \bigl( Ω¹_{\check X}|^{\vphantom{1}}_{\check S} \bigr) = \bigl[K_{\check
    X}|_{\check S} \bigr] = \bigl[ K_{\what X}|_{\what S} \bigr] \equiv 0
  $$
  and thus
  $$
  c_1 \bigl( Ω¹_{\check X}|^{\vphantom{1}}_{\check S} \bigr) ·
  (r^*H)|_{\check S} = c_1 \bigl( Ω¹_{\check X}|^{\vphantom{1}}_{\check S}
  \bigr)² - c_2 \bigl( Ω¹_{\check X}|^{\vphantom{1}}_{\check S} \bigr) = 0.
  $$
  Lemma~\ref{lem:7-1} thus applies to show that $(r_S)_* \bigl( Ω¹_{\check
    X}|^{\vphantom{1}}_{\check S} \bigr)$ is locally free and flat.  Set
  $Ω^{[1]}_X := \bigl( Ω¹_X \bigr)^{**}$.  Since $Ω^{[1]}_X|^{\vphantom{[1]}}_S$
  is reflexive by \cite[Prop.~5.2]{GKP13}, it necessarily agrees with $(r_S)_*
  \bigl( Ω¹_{\check X}|^{\vphantom{1}}_{\check S} \bigr)$, showing that both
  $Ω^{[1]}_X$ and $\sT_X = \bigl(Ω^{[1]}_X\bigr)^*$ are locally free near $S$.
  The Lipman-Zariski theorem for canonical spaces, \cite[Thm.~6.1]{GKKP11} or
  \cite[Thm.~3.8]{Druel13a}, thus applies, showing that $X$ is smooth near $S$.
\end{proof}
 
\begin{proof}[Proof of Theorem~\ref*{thm:tquot}]\label{pf:thm:tquot}
  Lemma~\ref{lem:7-2} asserts that $X$ is smooth in codimension two.  Under this
  additional assumption, Theorem~\ref{thm:tquot} has been shown in
  \cite[Thm.~1.16]{GKP13}.
\end{proof}

\end{document}